\documentclass[11pt,a4paper]{amsart}

\usepackage[utf8]{inputenc}
\usepackage[T1]{fontenc}
\usepackage[english]{babel}
\usepackage{amsmath,amssymb,amsfonts,amsthm,mathrsfs,mathtools}
\usepackage{enumitem}
\usepackage[colorlinks=true,linkcolor=black,urlcolor=blue,citecolor=black]{hyperref}
\usepackage{aliascnt}
\usepackage[noabbrev,capitalize]{cleveref}
\usepackage{fullpage}
\usepackage{xcolor}

\numberwithin{equation}{section}

\newtheorem{thm}{Theorem}[section]
\newaliascnt{prop}{thm}
\newtheorem{prop}[prop]{Proposition}
\aliascntresetthe{prop}
\newaliascnt{cor}{thm}
\newtheorem{cor}[cor]{Corollary}
\aliascntresetthe{cor}
\newaliascnt{lem}{thm}
\newtheorem{lem}[lem]{Lemma}
\aliascntresetthe{lem}
\newaliascnt{conj}{thm}

\aliascntresetthe{conj}
\theoremstyle{definition}
\newaliascnt{defi}{thm}

\aliascntresetthe{defi}
\newaliascnt{ex}{thm}

\aliascntresetthe{ex}
\theoremstyle{remark}
\newaliascnt{rmk}{thm}
\newtheorem{rmk}[rmk]{Remark}
\aliascntresetthe{rmk}

\crefname{thm}{Theorem}{Theorems}
\Crefname{thm}{Theorem}{Theorems}
\crefname{prop}{Proposition}{Propositions}
\Crefname{prop}{Proposition}{Propositions}
\crefname{cor}{Corollary}{Corollaries}
\Crefname{cor}{Corollary}{Corollaries}
\crefname{lem}{Lemma}{Lemmas}
\Crefname{lem}{Lemma}{Lemmas}
\crefname{conj}{Conjecture}{Conjectures}
\Crefname{conj}{Conjecture}{Conjectures}
\crefname{defi}{Definition}{Definitions}
\Crefname{defi}{Definition}{Definitions}
\crefname{ex}{Example}{Examples}
\Crefname{ex}{Example}{Examples}
\crefname{rmk}{Remark}{Remarks}
\Crefname{rmk}{Remark}{Remarks}

\newcommand{\Z}{\mathbb Z}

\newcommand{\R}{\mathbb R}

\newcommand{\T}{\mathbb T}
\newcommand{\E}{\mathbb E}
\newcommand{\PP}{\mathbb P}
\newcommand{\cF}{\mathcal F}
\newcommand{\1}{\mathbf 1}

\newcommand{\supp}{\operatorname{supp}}
\newcommand{\dist}{\operatorname{dist}}
\newcommand{\Cov}{\operatorname{Cov}}
\newcommand{\Var}{\operatorname{Var}}
\renewcommand{\Re}{\operatorname{Re}}
\renewcommand{\Im}{\operatorname{Im}}

\title[Fourier dimension of GMC]
{Fourier decay of Gaussian multiplicative chaos and boundary geometry}

\author{Marcu-Antone Orsoni}
\author{William Verreault}
\address{Marcu-Antone Orsoni, Universit\'e Laval, Qu\'ebec, QC, Canada}
\email{marcu-antone.orsoni@mat.ulaval.ca}
\address{William Verreault, University of Toronto, Toronto, ON, Canada}
\email{william.verreault@utoronto.ca}

\date{}

\begin{document}

\vspace*{-4mm}

\begin{abstract}
In this paper, we develop a new approach to studying the Fourier transform
of Gaussian multiplicative chaos. We determine the Fourier dimension of these random measures on tori for every smooth log-correlated covariance and in every dimension, and prove sharp bounds and exact formulas for
chaos on natural classes of bounded sets and hypersurfaces in Euclidean space. 
These results highlight how boundary geometry and multifractal concentration jointly govern Fourier decay, including regimes in which the boundary strictly reduces the Fourier dimension.
In the
circle setting, this gives a substantially simpler new proof of a conjecture of
Garban and Vargas.
\end{abstract}

\maketitle

\section{Introduction}

\subsection{Background}
Gaussian multiplicative chaos, introduced by Kahane \cite{Kahane1985}, is
the canonical framework for exponentiating a log-correlated Gaussian field.
Let $X$ be a centered Gaussian field on $\T^d$, or on an open set
$U\subset\R^d$, whose covariance is locally of the form
$$
 \E[X(x)X(y)]
 =
 \log\frac1{|x-y|}+g(x,y).
$$
In the smooth class considered in this paper, the covariance is smooth away
from the diagonal and $g$ is smooth near it.  Although $X$ is generally a
random distribution rather than a function, a suitable renormalization of
its exponential gives rise to a random measure $M_\gamma$ whenever
$
{0<\gamma<\sqrt{2d}}.
$
This is the subcritical Gaussian multiplicative chaos associated with $X$.
GMC has become a central object in modern probability, with a rich
geometric and analytic theory. We refer to the monograph
\cite{BerestyckiPowell} and the review \cite{RhodesVargas2014} for
background and applications. Its construction and the properties used in
this paper are recalled in \cref{sec:prelimGMC}.

A natural question in harmonic analysis is to ask how much cancellation occurs in
$$
 \widehat M_\gamma(\xi)
 =
 \int e^{-i\xi\cdot x}\,M_\gamma(dx)
$$
at high frequencies.  This question is particularly subtle because GMC is
singular and highly inhomogeneous.  Its local mass is governed by the thick
points of the underlying field, while its Fourier transform also records
oscillation across space.  The qualitative question is whether $M_\gamma$
is a Rajchman measure, meaning that its Fourier transform tends to zero.
The Fourier dimension quantifies the strongest uniform polynomial decay
that can hold.

The harmonic analysis of GMC was initiated by Garban and Vargas
\cite{GarbanVargas}.  Working with the canonical circle field, they proved
that its subcritical chaos is almost surely Rajchman, studied the limiting
fluctuations of its Fourier coefficients when $\gamma<1/\sqrt2$, and
conjectured that its Fourier dimension is equal to
$
 D_{\gamma,1}
$
where, in general,
$$
 D_{\gamma,d}
 =
 \begin{cases}
 d-\gamma^2,
 &0<\gamma^2\leq d/2,\\
 (\sqrt{2d}-\gamma)^2,
 &d/2<\gamma^2<2d.
 \end{cases}
$$
The quantity $D_{\gamma,d}$ is the almost sure correlation dimension of the
chaos, as follows from its multifractal spectrum formalism \cite{Bertacco2023}.  Since
correlation dimension bounds Fourier dimension from above, it is the
largest decay exponent one could expect.  The Garban--Vargas conjecture
therefore asks whether the spatial oscillation of the canonical circle
chaos is strong enough to saturate the obstruction created by its
multifractal concentration.  They also observed that their qualitative
Rajchman argument extends to the canonical fractional Gaussian field on
$\T^d$ in every dimension.

Lin, Qiu and Tan \cite{LinQiuTanI} first proved the conjectured equality for the exact
logarithmic kernel on $[0,1]$ and explained the corresponding modification
for the canonical circle field. They subsequently developed \cite{LinQiuTanII}  a
unified vector-valued martingale method which gives the exact Fourier
dimension in dimensions one and two, as well as the general lower bound
$
\min\{2,D_{\gamma,d}\}
$
under their higher-dimensional GMC hypotheses.  In
particular, this gives the exact value in the higher-dimensional parameter
range where $D_{\gamma,d}\leq2$. Chen, Lin and Qiu \cite{ChenLinQiu} later constructed, in
every dimension, a particular stationary torus field for which the exact
equality holds.

Our first theorem gives a new proof of the original circle conjecture and
proves its natural torus analogue for every smooth log-correlated covariance
in every dimension.
A dyadic version of our argument also recovers the particular model of Chen,
Lin and Qiu.  The finite-range reduction yields a short and direct proof
which applies uniformly in every dimension.
For the restriction of an ambient chaos to a general Euclidean domain, a
second mechanism enters the problem.  As in the classical Fourier analysis
of indicator functions, integration by parts converts bulk oscillation into
oscillatory integrals over the boundary. Their decay depends on whether
the boundary is flat or curved, each case producing different losses.  Our remaining results show that the Fourier dimension is governed by the interaction between this geometric
cancellation and the multifractal concentration of the chaos. A striking feature of our results is that the geometry of the
boundary can strictly lower the Fourier dimension, even though the
boundary itself carries zero GMC mass almost surely.

\subsection{Main results}

Our first result settles the torus problem throughout the smooth covariance
class.

\begin{thm}[Torus]
\label{thm:torus}
Let $d\geq1$, let $0<\gamma<\sqrt{2d}$ and let $M_\gamma$ be the GMC
associated with a smooth log-correlated Gaussian field on $\T^d$.  Then,
almost surely,
$$
 \dim_F(M_\gamma)=D_{\gamma,d}.
$$
\end{thm}

In particular, the theorem applies to the normalized fractional Gaussian
field
$$
 X(x)
 =
 a_d\sum_{m\in\Z^d\setminus\{0\}}
 \frac{\xi_m}{|m|^{d/2}}e^{im\cdot x},
$$
where $\xi_{-m}=\overline{\xi_m}$ and the variables indexed by one
representative of each pair $\{m,-m\}$ are independent standard complex
Gaussians.  The constant $a_d$ normalizes the logarithmic singularity, and
the series is interpreted in the sense of distributions.  For $d=1$, this
is the canonical circle field studied by Garban and Vargas, while for
$d=2$ it is a normalization of the mean-zero Gaussian free field on the
torus.
The particular block field of Chen, Lin and Qiu is also covered, although
not by smooth-perturbation invariance: its covariance remainder is only
known to be bounded and continuous.  Instead,
\cref{prop:dyadic-finite-range} applies directly to its dyadic
decomposition.

For the Euclidean results, let $U\subset\R^d$ be an open set on which the
field and its GMC are defined.  If $A\subset\R^d$ is a bounded set
with $A\Subset U$, we write
$$
 M_{\gamma,A}=\1_A M_\gamma.
$$
Here and below, $A\Subset U$ means that $\overline A$ is a compact subset
of $U$. Equivalently, $M_{\gamma,A}$ is the GMC constructed from $X$ with
reference measure $\1_A\,dx$. The ambient formulation ensures that the
logarithmic covariance remainder is smooth up to the boundary of $A$.
We reserve the letter $D$ for a
domain, by which we mean an open connected set, when the geometry of its
boundary enters the statement.  The first Euclidean result is
measure-theoretic and therefore belongs naturally to the more general
notation $A$.

\begin{thm}[Finite perimeter lower bound]
\label{thm:finite-perimeter}
Let $d\geq1$, let $U\subset\R^d$ be open and let
$A\subset\R^d$ be a bounded set
of positive Lebesgue measure and finite perimeter such that
$A\Subset U$.  Let $0<\gamma<\sqrt{2d}$ and let $M_\gamma$ be
the GMC associated with a smooth log-correlated Gaussian field on $U$.
Then, almost surely,
$$
 \dim_F(M_{\gamma,A})
 \geq
 \min\{2,D_{\gamma,d}\}.
$$
\end{thm}

The positive-measure assumption excludes the trivial case: if
$|A|=0$, then ${\E[M_\gamma(A)]=0}$ and hence
$M_{\gamma,A}=0$ almost surely.
The proof only uses finite perimeter through a
mollification estimate. A fractional version is recorded in
\cref{rmk:fractional-boundary} below.

Multiplication by a smooth compactly supported function introduces no
boundary loss.

\begin{cor}[Smooth localization]
\label{cor:smooth-localization}
Let $0<\gamma<\sqrt{2d}$, let $M_\gamma$ be associated with a smooth
log-correlated Gaussian field on an open set $U\subset\R^d$, and let
$0\leq\chi\in C_c^\infty(U)$ be nonzero. Then, almost surely,
$$
 \dim_F(\chi M_\gamma)=D_{\gamma,d}.
$$
\end{cor}

The positive Fourier decay also has a qualitative consequence under
absolute continuity, with no regularity assumption on the density or on
the boundary of a restricting set.

\begin{cor}[Rajchman property]
\label{cor:rajchman-domains}
Let $0<\gamma<\sqrt{2d}$ and let $M_\gamma$ be associated with a smooth
log-correlated Gaussian field on an open set $U\subset\R^d$.
Almost surely, every finite measure on $U$ that is absolutely
continuous with respect to $M_\gamma$ is Rajchman.
In particular, this holds simultaneously for all restrictions
$M_{\gamma,A}$ with $A\Subset U$.
\end{cor}

\begin{proof}
Choose a smooth exhaustion $V_j\Subset U$, with
$\overline V_j\subset V_{j+1}$ and $\bigcup_jV_j=U$.
By \cref{thm:finite-perimeter}, almost surely every
$M_{\gamma,V_j}$ is Rajchman. Work on this single probability-one event.
If $\mu$ is a compactly supported Rajchman measure and $f\in L^1(\mu)$,
choose $f_\ell\in C_c^\infty(\R^d)$ converging to $f$ in $L^1(\mu)$.
Each $f_\ell\mu$ is Rajchman because its Fourier transform is the
convolution of $\widehat f_\ell\in L^1(\R^d)$ with $\widehat\mu$, and
$$
 \sup_{\xi\in\R^d}
 |\widehat{f_\ell\mu}(\xi)-\widehat{f\mu}(\xi)|
 \leq\|f_\ell-f\|_{L^1(\mu)}.
$$
Thus $f\mu$ is Rajchman. Now let $\nu$ be any finite measure on $U$ absolutely continuous
with respect to $M_\gamma$, and set $\nu_j=\1_{V_j}\nu$.
Since $\nu_j\ll M_{\gamma,V_j}$, the preceding argument shows that
$\nu_j$ is Rajchman. Moreover,
$$
 \sup_{\xi\in\R^d}
 |\widehat\nu(\xi)-\widehat{\nu_j}(\xi)|
 \leq \nu(U\setminus V_j)\longrightarrow0.
$$
Thus $\nu$ is Rajchman. This deterministic argument applies on the
same probability-one event to every such $\nu$, proving the claim.
\end{proof}

The probability-one event is independent of the absolutely continuous
measure. In particular, the set in the last assertion may depend on
the realization of the chaos. 

We now turn to domains and the geometry of their boundaries.  When
$D_{\gamma,d}>2$, the exponent $2$ in
\cref{thm:finite-perimeter} is sharp in the presence of a planar
boundary piece. We say that $\partial D$ contains a \emph{flat open patch}
if a nonempty relatively open subset of $\partial D$ is contained in an
affine hyperplane. This is a local condition on the boundary. For a
convex domain, it is enough that a supporting face have positive
$(d-1)$-dimensional measure. We also call $D$ a \emph{Lipschitz} domain if its boundary is locally the graph of
a Lipschitz function, with $D$ lying on one side of that graph.
Note that every bounded Lipschitz domain has finite perimeter.

\begin{thm}[Flat open patch]
\label{thm:flat-boundary}
Let $d\geq2$, let $U\subset\R^d$ be open and let $D\Subset U$ be a bounded
Lipschitz domain whose boundary contains a flat open patch.  Let
$0<\gamma<\sqrt{2d}$ and let $M_\gamma$ be the GMC associated with a
smooth log-correlated Gaussian field on $U$.  Then, almost surely,
$$
 \dim_F(M_{\gamma,D})
 =
 \min\{2,D_{\gamma,d}\}.
$$
\end{thm}




Curvature creates additional tangential cancellation. For $p\geq1$, put
$$
 \alpha_{\gamma,p}
 =
 \big(\sqrt p-\frac{\gamma}{\sqrt2}\big)^2.
$$
We call a bounded $C^2$ convex domain \emph{uniformly strictly convex} if
the principal curvatures of its boundary are bounded below by a positive
constant.

\begin{thm}[Uniform strict convexity]
\label{thm:curved-domain}
Let $d\geq2$, let $U\subset\R^d$ be open and let $D\Subset U$ be a bounded
convex domain with a $C^\infty$ uniformly strictly convex boundary.  Let
$0<\gamma<\sqrt{2d}$ and let $M_\gamma$ be the GMC associated with a
smooth log-correlated Gaussian field on $U$.  Then, almost surely,
$$
 \dim_F(M_{\gamma,D})
 =
 \min\big\{D_{\gamma,d},2+\alpha_{\gamma,d-1}\big\}.
$$
\end{thm}


In dimensions
$2\leq d\leq4$, every uniformly strictly convex domain in
\cref{thm:curved-domain} satisfies
$
 \dim_F(M_{\gamma,D})=D_{\gamma,d}
$
throughout the subcritical range. In other words,
the geometric term $2+\alpha_{\gamma,d-1}$ can be active only when
$d\geq5$.  
For example, in dimension $d=5$ with $\gamma=1$, we have
$$
 D_{1,5}=4,\qquad
 \dim_F(M_{1,B})=2,\qquad
 \dim_F(M_{1,D})=\frac{13}{2}-2\sqrt2\approx 3.672,
$$
where $B$ is a box and $D$ is uniformly strictly convex. Thus both boundary
geometries produce a strict loss, of different sizes, relative to the
smoothly localized chaos.

The surface estimate in the proof also determines the Fourier dimension
of genuine GMC constructed with surface measure on the boundary.

\begin{thm}[GMC on a uniformly strictly convex hypersurface]
\label{thm:curved-surface}
Let $d\geq2$, let $U\subset\R^{d}$ be open, and let
$S\Subset U$ be the $C^\infty$ uniformly strictly convex
boundary of a bounded convex domain. Let $M_\gamma^S$ be the GMC
constructed from a smooth log-correlated Gaussian field on $U$ and
surface measure on $S$. If $0<\gamma<\sqrt{2(d-1)}$, then, almost surely,
$$
 \dim_F(M_\gamma^S)
 =\alpha_{\gamma,d-1}
 =\Big(\sqrt{d-1}-\frac{\gamma}{\sqrt2}\Big)^2.
$$
\end{thm}

Here Fourier decay is measured in the ambient Euclidean space $\R^{d}$. In
particular, for circular chaos the angular Fourier dimension is
$D_{\gamma,1}$, whereas the Euclidean Fourier dimension on
$S^1\subset\R^2$ is $(1-\gamma/\sqrt2)^2$. The same measure therefore
exhibits different exponents for the linear angular phase and the
Euclidean phase restricted to the circle.

\begin{rmk}
These results have an automatic
consequence for the full Fourier spectrum $ \dim_F^\theta(M_{\gamma,D})$, whose definition is recalled in \cref{sec:prelimDim}. Its endpoints are
the Fourier dimension at $\theta=0$ and the Sobolev dimension at $\theta=1$. 
Whenever the Fourier dimension
reaches the correlation obstruction, combining
\eqref{eq:reldim}, the correlation dimension formula and monotonicity in $\theta$ yields
$$ 
\dim_F^\theta(\mu)=D_{\gamma,d},
 \qquad 0\leq\theta\leq1.
$$
This applies to the torus fields of \cref{thm:torus}, the dyadic
fields of \cref{prop:dyadic-finite-range}, and the smooth
localizations of \cref{cor:smooth-localization}. It also applies to the restrictions $M_{\gamma,A}$ in
\cref{thm:finite-perimeter} when $A$ has nonempty interior and
$D_{\gamma,d}\leq2$, and to the uniformly strictly convex domains of
\cref{thm:curved-domain} when
$D_{\gamma,d}\leq2+\alpha_{\gamma,d-1}$.
When a boundary obstruction is active, concavity still gives the lower bound
$$
 \dim_F^\theta(M_{\gamma,D})
 \geq
 (1-\theta)\dim_F(M_{\gamma,D})+\theta D_{\gamma,d}.
$$
Moreover, for domains with a flat open patch, a small modification of the proof of \cref{thm:flat-boundary} shows that,
simultaneously for $0\leq\theta\leq1$,
$$
 \dim_F^\theta(M_{\gamma,D})
 \leq
 \min\{D_{\gamma,d},2+(d-1)\theta\}.
$$
Determining the full Fourier spectrum for more general domains remains an interesting open problem, which we do not pursue here.
\end{rmk}

\subsection{Comparison with previous work and related models}
An early quantitative result on the Fourier transform of GMC was obtained
by Falconer and Jin \cite{FalconerJin2019}.  They proved a positive
Fourier dimension bound for Liouville quantum gravity on a rotund convex
planar domain in a small parameter regime. Their chaos is constructed from
the zero boundary Gaussian free field.  Our domain results instead concern
the restriction of a smooth ambient log-correlated field defined across the
boundary.  The two settings are therefore complementary.

The torus theorem extends the circle equality of
\cite{LinQiuTanI,LinQiuTanII} to every prescribed smooth covariance in
every dimension. The separate dyadic criterion also covers the bounded
continuous covariance remainder of the construction in \cite{ChenLinQiu}.
For Euclidean cubes in the smooth ambient class,
\cref{thm:flat-boundary} shows that the lower bound
$\min\{2,D_{\gamma,d}\}$ of \cite{LinQiuTanII} is sharp.
\cref{thm:finite-perimeter} extends this lower bound to arbitrary
bounded sets of finite perimeter in the same covariance class, while
\cref{thm:curved-domain} determines the Fourier dimension for
uniformly strictly convex domains.

The random exponents appearing here have a standard interpretation in GMC
theory. The quantity $D_{\gamma,d}$ is obtained from the $L^q$-spectrum, while $\alpha_{\gamma,p}$ is the lower endpoint
of the singularity spectrum \cite{Bertacco2023}.
This multifractal interpretation does not, however, subsume our domain
results. Restricting an ambient chaos to domains with nonempty interior
does not change its bulk $L^q$-spectrum, yet flat and curved boundaries can
have different Fourier dimensions. 
The Fourier dimension depends both on the distribution of the
mass and on the oscillatory geometry of the support.

There is a direct deterministic precedent for this phenomenon.  For
Lebesgue measure $\1_D\,dx$, the Gauss--Green formula contributes the
baseline factor $|\xi|^{-1}$.  A planar face makes this estimate sharp in
its normal direction, giving Fourier dimension $2$.  If instead $D$ is a
smooth convex body with everywhere positive Gaussian curvature, stationary
phase gives the sharp decay
$$
 |\widehat{\1_D}(\xi)|
 \lesssim
 |\xi|^{-(d+1)/2},
$$
and hence Fourier dimension $d+1$.
These estimates are classical (see, for example, \cite{Herz1962}, or \cite{BrunaNagelWainger} for cases with more complicated geometry). The flat and uniformly strictly
convex theorems may therefore be viewed as random multifractal counterparts
of the two classical geometric endpoints. The boundary terms $2$ and
$2+\alpha_{\gamma,d-1}$ are random analogues of the classical flat and
curved exponents: formally setting $\gamma=0$ gives $2$ and $d+1$. The
full curved formula does not extend continuously to $\gamma=0$, since the
multifractal bulk obstruction
disappears for Lebesgue measure.

Fourier decay has also been studied for non-Gaussian multiplicative
measures.  Shmerkin and Suomala \cite{ShmerkinSuomala} obtained polynomial Fourier decay and
Salem-type conclusions for classes of
spatially independent martingales, including fractal percolation measures. For canonical Mandelbrot cascades, exact
Fourier dimension formulas were subsequently obtained independently, under
different moment hypotheses, by Chen, Li and Suomala \cite{ChenLiSuomala} and by Chen, Han, Qiu
and Wang
\cite{ChenHanQiuWang}. The latter work introduced the
vector-valued martingale viewpoint later developed in
\cite{LinQiuTanI,LinQiuTanII}. The unified theory of Lin, Qiu and Tan also
gives exact results for canonical Mandelbrot random coverings and
nontrivial lower bounds for Poisson multiplicative chaos and generalized
Mandelbrot cascades.

A particularly close geometric analogue is provided by Mandelbrot
cascades pushed forward to planar curves of nonvanishing curvature.
Under suitable moment hypotheses, their Fourier dimension is almost
surely, on the event that the limiting measure is nonzero, equal to the infimum of their lower local
dimensions \cite{CaiFangQu2026, RyouSuomala2026}.
\cref{thm:curved-surface} gives the corresponding formula for
GMC on uniformly strictly convex hypersurfaces in every dimension.
For the domain restrictions $M_{\gamma,D}$, the ambient chaos assigns
zero mass to $\partial D$. Surface chaos nevertheless enters the proof
through the boundary terms obtained by normal integration by parts.
The surface estimate thus also explains the boundary obstruction
in \cref{thm:curved-domain}.

Other phases of GMC lead to different Fourier questions. Critical
one-dimensional GMC, corresponding to $\gamma=\sqrt2$, has Hausdorff (and hence Fourier) dimension zero \cite{BarralKupiainenNikulaSaksmanWebb}. The meaningful question is therefore qualitative or logarithmic decay. Garban
and Vargas \cite{GarbanVargas} explicitly posed the almost sure Rajchman property as an open
problem. For critical GMC on $[0,1]$ generated by a
$\star$-scale invariant field,
Arguin and Hamdan \cite{ArguinHamdan} proved
that
$
 (\log n)^\alpha\widehat M_{\sqrt2}(n)\to0
$
in probability for every $\alpha<1/4$.  Very recently,
Cai, Chen, Fang and Guo proved the almost sure Rajchman property for the
canonical critical circle chaos \cite{CaiChenFangGuo2026}.

In the purely imaginary regime, the chaos is a complex-valued random
distribution rather than a positive measure.  Bonnefont, Rajam\"{a}ki and
Vargas \cite{BonnefontRajamakiVargas} extended the notion of Fourier dimension to this setting and proved,
for the circle imaginary chaos with $0<\beta<1$, that
$$
 \dim_F(M_{\mathrm{i}\beta})=1-\beta^2
$$
almost surely.  Their dimension formula is
stable under sufficiently regular stationary and nonstationary covariance
perturbations.  For the exact circle kernel they also prove a central limit
theorem for the optimally rescaled coefficients.

A particularly close comparison is provided by holomorphic multiplicative
chaos \cite{NajnudelPaquetteSimm}.  For $0<\theta\leq1$, define its
coefficients by
$$
 \exp\Big(
 \sqrt\theta\sum_{k\geq1}\frac{X_k}{\sqrt k}z^k
 \Big)
 =
 \sum_{n\geq0}A_n^{(\theta)}z^n,
$$
where the $X_k$ are independent standard complex Gaussians. Under the parameter correspondence $\theta=2\gamma^2$, the coefficient
$A_n^{(\theta)}$ is naturally compared with the rescaled real GMC
Fourier coefficient $n^{\gamma^2/2}\widehat M_\gamma(n)$.
This correspondence extends to the limiting distribution when
$\gamma<1/\sqrt2$, equivalently $\theta<1$.  After their respective
deterministic normalizations, both coefficients converge to a complex
Gaussian mixture whose random variance is the total mass of an associated
GMC.  In the notation of real GMC, the random variance is the total mass
$M_{2\gamma}(\T)$.
The real GMC convergence is proved by Garban and Vargas
\cite{GarbanVargas}, while the corresponding HMC convergence throughout
the full range $0<\theta<1$ is established in
\cite{NajnudelPaquetteSimmVu}.

At the critical HMC parameter $\theta=1$, write $A_n=A_n^{(1)}$.  These
coefficients arise as the fixed-index large-matrix limits of secular
coefficients of the circular unitary ensemble.  Atherfold and Najnudel \cite{AtherfoldNajnudel}
prove that $(\log n)^{1/4}A_n$ converges in distribution to a complex
Gaussian mixture whose random variance is the total mass of critical GMC.
Under the aforementioned correspondence, this critical HMC point matches
the Fourier freezing threshold $\gamma=1/\sqrt2$ for real GMC, rather than
the critical real GMC parameter $\gamma=\sqrt2$.

The same critical HMC coefficients also form the model problem studied by
Soundararajan and Zaman \cite{SoundararajanZaman} in connection with random
multiplicative functions.  In the arithmetic setting, Harper
\cite{Harper2020} identified the analogous critical chaos structure in the
low moments of normalized partial sums of random multiplicative functions.
The same critical chaos energy mechanism enters the sharp almost sure
estimates in \cite{Verreault2026}.  Gorodetsky and Wong \cite{GorodetskyWong} obtain analogous
Gaussian-mixture limits for random multiplicative functions.

The GMC--HMC correspondence is therefore quantitative and reaches the
limiting distribution, although no general transfer principle directly
identifies the two sequences.  The connection with random multiplicative
functions remains more structural.  
It would be interesting to
determine whether the finite-range viewpoint developed here can also be
used to study distributional asymptotics of GMC Fourier coefficients. Our methods are robust and can be adapted to other settings such as critical or imaginary GMC as described above. This is part of ongoing work.

Finally, circle GMC also appears in the probabilistic construction of
the Virasoro generators in Liouville conformal field theory
\cite{BGKRV}. For $0<\gamma<\sqrt2$, the $n$th generator differs
from its free-field counterpart by a multiplication operator whose
coefficient is proportional to $\widehat M_\gamma(-n)$, with a factor
independent of $n$.
\cref{thm:torus} therefore
identifies the sharp almost sure polynomial threshold of these random
scalar coefficients after the boundary field is fixed. This gives
quantitative content to the observation of Garban and Vargas that the high
Virasoro modes approach their free-field counterparts
\cite{GarbanVargas}. For $\sqrt2\leq\gamma<2$, the interaction coefficients are instead
defined through Girsanov shifts at the level of quadratic forms,
and their high-mode behaviour is a different but interesting question.

\subsection{Proof overview}
Our proofs use the standard $\star$-scale-invariant construction (see, e.g.,
\cite[Section~2.2]{RhodesVargas2014}), but with a smooth compactly supported
positive-definite kernel, and its periodization on the torus. The fields
have independent increments over disjoint scale intervals. The scale
parameter is continuous, which
will also allow us to reveal the field progressively in the curved
boundary argument. It suffices to work with this model since adding suitable independent smooth Gaussian
fields makes its covariance agree with the prescribed covariance.
The corresponding chaos measures are multiplied by smooth positive
functions, which preserves Fourier dimension.

The correlation dimension formula recalled in \cref{sec:prelimGMC}
gives the upper bound $D_{\gamma,d}$ on the torus and for domain
restrictions. The main task is therefore to prove matching lower bounds,
except when the boundary creates a smaller obstruction. At frequencies
$|\xi|\asymp2^n$, we choose a cutoff $\delta$ slightly larger than
$2^{-n}$ and write, schematically,
$$
 M_\gamma
 =
 \underbrace{M_{>\delta}}_{\text{Coarse}}
 +
 \underbrace{(M_\gamma-M_{>\delta})}_{\text{Fine}}.
$$

\vspace{-0.3em}

On the torus, the coarse density is smooth, and its derivative estimates
allow repeated integration by parts to make its contribution negligible.
For the fine fluctuation, the compact support of the kernel gives
independence between the restrictions of the entire fine field to regions
separated by more than $\delta$. Partitioning space into cubes of side comparable to $\delta$
therefore gives a fixed number of conditionally independent families.
At this single cutoff, we estimate their sums in weighted Fourier spaces
with a conditional independent-sum inequality.
Local GMC moment bounds give summable estimates over annuli,
which is enough to obtain almost sure
Fourier decay. Optimizing the moment exponent gives the lower bound
$D_{\gamma,d}$.

For a Euclidean restriction $M_{\gamma,A}$, the fine fluctuation is
estimated in the same way. The new difficulty is the factor $\1_A$ in
the coarse density. If $A$ has finite perimeter, mollification
approximates $\1_A$ with an $L^1$ error of order $\delta$.
The smooth approximation can again be treated by integration by parts, which
together with the fine estimate gives the lower bound
$\min\{2,D_{\gamma,d}\}$.

For boundaries with a flat open patch, we need to prove that the exponent
$2$ is sharp when $D_{\gamma,d}>2$. The mechanism is clearest for the
cube $[0,1]^d$. If its GMC had Fourier dimension greater than $2$, its
normal projection would have an integrable Fourier transform, hence a
continuous density which vanishes at the endpoint $0$. Its one-sided
averages at $0$ would therefore tend to zero. On the other hand, these
averages have constant positive expectation and, by comparison with the
GMC on the corresponding face, a uniform $L^q$ bound for some $q>1$.
They are therefore uniformly integrable and cannot converge to zero
almost surely. This contradiction is local,
and a smooth cutoff reduces a general flat open patch to exactly the same
argument.


For uniformly strictly convex domains, normal integration by parts
reduces the coarse term to oscillatory surface integrals, with a factor
$|\xi|^{-1}$ from the first integration. The surface phase has two
nondegenerate stationary points. To estimate these integrals, we reveal
the field continuously in scale while expanding a cutoff around the
stationary points. The motion of the cutoff contributes a term supported
where the phase oscillates, which is handled by tangential integration
by parts. The remaining martingale is controlled through its quadratic
variation, using properties of our finite-range model and local surface mass estimates.
This gives the surface exponent $\alpha_{\gamma,d-1}$ and hence the
boundary threshold $2+\alpha_{\gamma,d-1}$. The same surface estimate
also proves the lower bound in \cref{thm:curved-surface}.

The matching curved upper bound follows by testing Fourier decay on
thin caps cut out by supporting hyperplanes. Such decay imposes a
uniform upper bound on their masses. Large values of the coarse field
along the boundary, combined with the independent fine chaos, instead
produce caps too heavy for any Fourier exponent greater than
$2+\alpha_{\gamma,d-1}$, with positive probability. Removing finitely
many scales only multiplies the chaos by a smooth positive function,
so its Fourier dimension is a tail quantity. Kolmogorov's zero--one law
makes this obstruction almost sure. Applying the same coarse--fine
argument to small surface balls gives the surface upper bound
$\alpha_{\gamma,d-1}$.

\subsection{Further geometric questions}
The flat and uniformly strictly convex cases do not exhaust all smooth convex
boundaries. For a simple intermediate example, let $m\geq4$ be even and
consider
$$
 D_m
 =
 \big\{(x',x_d)\in\R^{d-1}\times\R:
 |x'|^m+x_d^2<1\big\}.
$$
The boundary has no flat open patch, but its curvature vanishes at the two
points $(0,\pm1)$. Near either of these points, after translation and
rotation, the boundary is of the form
$$
 x_d=-\frac12|x'|^m+O(|x'|^{2m}).
$$
Thus the contact with the supporting hyperplane has finite order $m$ in
every tangential direction. This is the isotropic finite-type model
familiar from the deterministic theory
\cite{BrunaNagelWainger,Schulz1991}. More generally, one could allow finitely many
isolated points at which the order of contact with the supporting
hyperplane is finite and the same in every tangential direction.
The supporting-cap argument in \cref{sec:curved-upper} suggests that
$$
 \dim_F(M_{\gamma,D_m})
 =
 \min\big\{
 D_{\gamma,d},
 2+\alpha_{\gamma,d-1},
 2+\frac{2(d-1)+\gamma^2}{m}
 \big\}
$$
almost surely. Indeed, at frequency $R$, the cap around either exceptional
point has normal thickness $R^{-1}$ and tangential radius $R^{-1/m}$,
so its deterministic volume is $R^{-1-(d-1)/m}$.  Since the exceptional
point is fixed, the normalization of the coarse field at tangential scale
$R^{-1/m}$ contributes $R^{-\gamma^2/(2m)+o(1)}$.  This leads to the last
exponent above.

This example suggests a broader dependence of Fourier dimension
on the order of contact with supporting hyperplanes.
Establishing the conjectural formula would in particular require
a surface estimate uniform in frequency as the stationary point
approaches a point of vanishing curvature.

Other geometric questions include domains with several curved boundary
components, where an inner boundary need not define a supporting
hyperplane for the whole measure, and the Fourier spectrum of uniformly
strictly convex restrictions when the boundary obstruction is active.
The arguments also suggest two refinements beyond dimension:
the optimal logarithmic correction to the torus decay at
$D_{\gamma,d}$, and an almost sure asymptotic for the largest supporting-cap
mass. The latter would strengthen the heavy-cap estimate used here, but
requires uniform control of unusually small fine chaos masses as well as
large coarse field values.

\subsection{Organization of the paper}

\cref{sec:2} develops the analytic and probabilistic tools used throughout the
paper, including the finite-range model reduction.  \cref{sec:3} proves the torus theorem and concludes with a dyadic
finite-range criterion that covers, in particular, the field of Chen, Lin
and Qiu. \cref{sec:4} treats Euclidean restrictions, proving the general
finite perimeter lower bound, smooth localization, and the sharp result
for domains whose boundary contains a flat open patch. \cref{sec:5}
proves the exact results for uniformly strictly convex domains and their
surface chaos.

Readers interested only in the torus problem may proceed from the relevant
preliminaries in \cref{sec:2} directly to \cref{sec:3}, whose core argument is
contained in a few pages.

\begin{rmk}
This project was initiated before the works of Lin, Qiu and Tan \cite{LinQiuTanI,LinQiuTanII}, and Chen, Lin and Qiu
\cite{ChenLinQiu} on the conjecture of Garban and Vargas and on the Fourier dimension of GMC on tori appeared. An earlier version of this work was
already presented in April and May 2025 in seminars at the University of
Science and Technology of China, Nanjing University, and NYU Shanghai.
Substantial improvements to the results have been obtained since.
\end{rmk}

\section{Preliminaries} \label{sec:2}

\subsection{Notation}
For variables $a$ and $b$, we write $a\lesssim b$ or $a=O(b)$ if
$|a|\leq C|b|$ for a constant $C$ independent of the parameters under
consideration.  If the constant depends on a parameter $k$, we write
$a\lesssim_k b$ or $a=O_k(b)$.  Constants may also depend on the fixed
field and geometric data.  We write $a\asymp b$ if both $a\lesssim b$ and
$b\lesssim a$ hold.

On metric spaces, all sets and functions are understood to be Borel
measurable, and all measures are defined on the Borel $\sigma$-algebra.
For $u>0$, we write
$
 \log_+u=\max\{\log u,0\},
$
and set $\log_+0=0$.  Euclidean and multi-index norms are both denoted by
$|\cdot|$. Indicator functions for a set $A$ are denoted by $\1_A$. We write $\T^d=(\R/2\pi\Z)^d$.  If $\mu$ is a finite measure
on $\T^d$, its Fourier coefficients are
$$
 \widehat\mu(m)
 =
 \int_{\T^d}e^{-im\cdot x}\,\mu(dx),
 \qquad m\in\Z^d.
$$
For a compactly supported finite measure on $\R^d$, we use
$$
 \widehat\mu(\xi)
 =
 \int_{\R^d}e^{-i\xi\cdot x}\,\mu(dx),
 \qquad \xi\in\R^d.
$$

\subsection{Fractal dimensions and weighted Fourier norms}
\label{sec:prelimDim}

We recall the notions of dimension used below.  General references are
\cite{Falconer1997,Falconer2014,Mattila2015}.  Throughout this subsection,
$\mu$ denotes a nonzero finite compactly supported measure.  We
present the definitions on $\R^d$.  On $\T^d$, Hausdorff and correlation
dimensions are defined using the quotient metric
$$
 d_{\T^d}(x,y)
 =
 \min_{\ell\in(2\pi\Z)^d}|x-y+\ell|,
$$ while Fourier and
Sobolev quantities use Fourier coefficients and sums over $\Z^d$. 

The (lower) Hausdorff dimension of $\mu$ is
$$
 \dim_H(\mu)
 =
 \inf\{
 \dim_H(E):
 E\subset\R^d,\ \mu(E)>0
 \}.
$$
For $s>0$, the Riesz $s$-energy of $\mu$ is
$$
 I_s(\mu)
 =
 \iint_{\R^d\times\R^d}
 \frac{\mu(dx)\mu(dy)}{|x-y|^s}.
$$
When $0<s<d$, it has the Fourier representation
$$
 I_s(\mu)
 \asymp_{d,s}
 \int_{\R^d}
 |\widehat\mu(\xi)|^2|\xi|^{s-d}\,d\xi,
$$
see \cite[Theorem~3.10]{Mattila2015}.  This motivates the Sobolev
dimension
$$
 \dim_S(\mu)
 =
 \sup\big\{
 s\in\R:
 \int_{\R^d}
 |\widehat\mu(\xi)|^2
 (1+|\xi|)^{s-d}\,d\xi<\infty
 \big\}.
$$
For $s\geq d$, the integral is a Sobolev regularity condition rather than
the Fourier representation of a Riesz energy.

The Fourier dimension of $\mu$ is
$$
 \dim_F(\mu)
 =
 \sup\big\{
 s\geq0:
 \sup_{\xi\in\R^d}
 (1+|\xi|)^{s/2}|\widehat\mu(\xi)|<\infty
 \big\}.
$$
We do not truncate the Fourier or Sobolev dimension of a measure at the
ambient dimension.  This convention is compatible with the Fourier
spectrum below.  Replacing $|\xi|$ by $1+|\xi|$ only changes the
contribution of bounded frequencies.

The lower correlation dimension is
$$
 \dim_2(\mu)
 =
 \liminf_{r\to0}
 \frac{
\log\displaystyle\int_{\R^d}\mu(B(x,r))\,\mu(dx)
 }{\log r}.
$$
Equivalent definitions using cubes or smooth approximate identities are
standard.  
To relate this
quantity to energy, put
$$
 C_\mu(r)
 =
 \int_{\R^d}\mu(B(x,r))\,\mu(dx).
$$
Tonelli's theorem gives, for every $s>0$,
$$
 I_s(\mu)
 =
 s\int_0^\infty r^{-s-1}C_\mu(r)\,dr.
$$
If $I_s(\mu)<\infty$, then
$
 C_\mu(r)\leq r^sI_s(\mu),
$
and hence $\dim_2(\mu)\geq s$.  Conversely, if
$s<\dim_2(\mu)$, choose $t$ with
$
 s<t<\dim_2(\mu).
$
Then $C_\mu(r)\leq r^t$ for every sufficiently small $r$, and the
preceding integral representation gives $I_s(\mu)<\infty$.  These two
implications show that
$$
 \dim_2(\mu)
 =
 \sup\big(\{s>0:I_s(\mu)<\infty\}\cup\{0\}\big).
$$
Moreover, this supremum never exceeds $d$.  
The Fourier representation of Riesz energy now gives
$$
 \dim_2(\mu)
 =
 \min\{\dim_S(\mu),d\}.
$$
Together with the energy criterion for Hausdorff dimension
\cite[Section~2.5]{Mattila2015}, this gives
\begin{equation}
\label{eq:reldim}
 \min\{\dim_F(\mu),d\}
 \leq
 \dim_2(\mu)
 =
 \min\{\dim_S(\mu),d\}
 \leq
 \dim_H(\mu).
\end{equation}
The definitions also give
$
 \dim_F(\mu)\leq\dim_S(\mu).
$
For $0<s<d$, the same energy relations hold on $\T^d$. 

The estimates below naturally involve the Fourier spectrum. Following
\cite{Fraser2024}, for $\theta\in(0,1]$ and $s\geq0$, set
$$
 \mathcal J_{s,\theta}(\mu)
 =
 \Big(
 \int_{\R^d}
 |\widehat\mu(\xi)|^{2/\theta}
 (1+|\xi|)^{s/\theta-d}\,d\xi
 \Big)^\theta,
$$
and
$$
 \mathcal J_{s,0}(\mu)
 =
 \sup_{\xi\in\R^d}
 |\widehat\mu(\xi)|^2(1+|\xi|)^s.
$$
The Fourier spectrum of $\mu$ is
$$
 \dim_F^\theta(\mu)
 =
 \sup\big\{s\geq0:\mathcal J_{s,\theta}(\mu)<\infty\big\}.
$$
Thus
$$
 \dim_F^0(\mu)=\dim_F(\mu),
 \qquad
 \dim_F^1(\mu)=\dim_S(\mu).
$$
The function
$
 \theta\mapsto\dim_F^\theta(\mu)
$
is nondecreasing and concave on $[0,1]$ by
\cite[Theorem~1.1]{Fraser2024}.
The same proof applies on $\T^d$ with counting measure on $\Z^d$.

For $n\geq0$, set
$$
 A_n^\T
 =
 \{m\in\Z^d:2^n\leq|m|<2^{n+1}\},
 \qquad
 A_n^\R
 =
 \{\xi\in\R^d:2^n\leq|\xi|<2^{n+1}\}.
$$
For every positive integer $k$ and $s\in\R$, define
$$
 \|f\|_{B_{k,s,n}^\T}^{2k}
 =
 \sum_{m\in A_n^\T}
 |f(m)|^{2k}|m|^{ks-d},
 \qquad
 \|f\|_{B_{k,s,n}^\R}^{2k}
 =
 \int_{A_n^\R}
 |f(\xi)|^{2k}|\xi|^{ks-d}\,d\xi.
$$
We omit the superscript when the setting is clear.  The exponent $ks-d$
is the Fourier spectrum weight corresponding to $\theta=1/k$.

The following elementary criterion records how these annular norms control
the Fourier spectrum and pointwise Fourier decay.  Related moment criteria
for random measures and discrete $L^p$ formulations of Fourier spectrum
energies of compactly
supported measures appear in \cite{CarnovaleFraserdeOrellana2024} and \cite{Fraser2024}.

\begin{lem}[Annular Fourier criterion]
\label{lem:weighted-fourier-criterion}
Let $\mu$ be a finite measure on $\T^d$, or a compactly supported finite
measure on $\R^d$.  If
$$
 \sum_{n\geq0}
 \|\widehat\mu\|_{B_{k,s,n}}^{2k}<\infty
$$
for some positive integer $k$ and $s\geq0$, then
$$
 \dim_F^{1/k}(\mu)\geq s
 \qquad\text{and}\qquad
 \dim_F(\mu)\geq s-\frac{d}{k}.
$$
Consequently, if the displayed estimate holds for an unbounded sequence of
integers $k$, then $\dim_F(\mu)\geq s$.
\end{lem}

\begin{proof}
The first conclusion follows by decomposing
$\mathcal J_{s,1/k}(\mu)$ into dyadic annuli.  The contribution of bounded
frequencies is finite because $\mu$ is finite.

On the torus, every term in the annular sum is bounded by the full sum.
It follows that
$$
 |\widehat\mu(m)|
 \lesssim_{\mu,k,s}
 |m|^{-s/2+d/(2k)},
 \qquad m\neq0.
$$

In the Euclidean case, the hypothesis gives
$$
 \int_{\R^d}
 |\widehat\mu(\xi)|^{2k}
 (1+|\xi|)^{ks-d}\,d\xi<\infty.
$$
Choose $\chi\in C_c^\infty(\R^d)$ equal to one on $\supp\mu$.  Since
$\mu=\chi\mu$, we have $\widehat{\mu}=(2\pi)^{-d}\widehat\chi*\widehat\mu$. For $\beta\in\R$, the elementary estimate
\begin{equation} \label{eq:elemes}
 (1+|\xi|)^\beta
 \leq(1+|\xi-\eta|)^\beta(1+|\eta|)^{|\beta|}
\end{equation}
holds. Coupling it with H\"older's inequality therefore gives
$$
 \sup_{\xi\in\R^d}
 (1+|\xi|)^\beta|\widehat\mu(\xi)|
 \lesssim_{\chi,\beta,k}
 \Big(
 \int_{\R^d}
 |\widehat\mu(\xi)|^{2k}
 (1+|\xi|)^{2k\beta}\,d\xi
 \Big)^{1/(2k)}.
$$
Taking
$
 \beta=s/2-d/(2k)
$
makes the weighted integral over the unit ball finite because $\mu$ is
finite, while its contribution outside the unit ball is comparable to the
annular sum.  This gives the required pointwise decay.
\end{proof}

\subsection{A crash course on Gaussian multiplicative chaos}
\label{sec:prelimGMC}

Gaussian multiplicative chaos, introduced by Kahane \cite{Kahane1985} in connection with
Mandelbrot's model of turbulence, gives a rigorous meaning to the formal
random measure $e^{\gamma X(x)}\,dx$.  Much of its modern
development was driven by Liouville quantum gravity, where it constructs
the random area measure arising from Polyakov's formulation of
two-dimensional quantum gravity
\cite{DuplantierSheffield2011,Polyakov1981,RhodesVargas2014}. It has since
become a central object in random geometry and conformal field theory, with
applications to random matrix theory and analytic number theory
\cite{BerestyckiPowell,BerestyckiWebbWong2018,SaksmanWebb2020}.

Let $X$ be a centered Gaussian field on $\T^d$, or on an open set
$U\subset\R^d$.  We call $X$ \emph{smooth log-correlated} if its covariance
kernel $K$ is smooth away from the diagonal and, locally near the diagonal,
$$
 K(x,y)=\log\frac1{|x-y|}+g(x,y),
$$
where $g$ is smooth.  On the torus, Euclidean distance is replaced locally
by torus distance.

Regularity assumptions on the covariance remainder vary in the literature.
A common general formulation on a bounded domain $V$ assumes
$g\in C(\overline V\times\overline V)$, and hence that $g$ is bounded.  The
Fourier arguments in this paper require the stronger smoothness assumption
above.

The word \emph{locally} is important, since in a sufficiently small coordinate
neighbourhood,
$
 \log(1/|x-y|)=\log_+(1/|x-y|).
$
The ordinary logarithm emphasizes the local singularity and allows the
remainder $g$ to be smooth.  The truncated logarithm is more convenient in
global covariance estimates.  On a fixed compact set the two conventions
differ by a bounded function, which is enough for our purposes. 
The hard cutoff in $\log_+$ is not smooth, however, so a
global decomposition with a bounded or continuous remainder is weaker than
the smooth decomposition assumed in our Fourier theorems.

Let $X_\varepsilon$ be a standard convolution regularization of $X$ (see, e.g., \cite[Section~3.2]{BerestyckiPowell}) and set
$$
 M_{\gamma,\varepsilon}(dx)
 = e^{\gamma X_\varepsilon(x)
 -\frac{\gamma^2}{2}\E[X_\varepsilon(x)^2]}\,dx.
$$
When $0<\gamma^2<2d$, these measures converge in probability, for the vague topology on Radon measures, to a nontrivial measure $M_\gamma$.  Moreover,
$M_{\gamma,\varepsilon}(f)\to M_\gamma(f)$ in $L^1$ for every
$f\in C_c(U)$, and the limit does not depend on the usual choice of
regularization.  This is the \emph{subcritical Gaussian multiplicative
chaos} associated with $X$.  We refer to
\cite{BerestyckiPowell,Kahane1985,RhodesVargas2014} for its
construction and basic properties.  At $\gamma^2=2d$, the normalization
above becomes degenerate and must be replaced by a critical normalization.

If $A\subset U$ is a bounded set with $\overline A\subset U$, we
write
$$
 M_{\gamma,A}=\1_A M_\gamma.
$$
If $S\subset U$ is a smooth $p$-dimensional submanifold, with
$1\leq p<d$, we denote by
$M_\gamma^S$ the GMC constructed from the restricted covariance
$K|_{S\times S}$ and the surface measure $\sigma_S$.  This surface chaos is
not the restriction of the ambient volume chaos to $S$, which almost surely
assigns zero mass to $S$.  Indeed,
$
 \E[M_\gamma(S)]=\operatorname{Leb}_d(S)=0.
$
The surface chaos is subcritical when
$\gamma^2<2p$.

For a $p$-dimensional chaos, the multifractal spectrum, or power-law
exponent, is
$$
 \zeta_p(q)
 =
 \big(p+\frac{\gamma^2}{2}\big)q
 -\frac{\gamma^2}{2}q^2.
$$
The following standard estimate is the scalar moment input used throughout.
We use the convention $M_{\gamma,0}=M_\gamma$.  In the surface case, the
notation denotes the corresponding surface chaos.

\begin{prop}[GMC moments]
\label{prop:gmcmoments}
Let $p\geq1$ be an integer, let $0<\gamma^2<2p$ and let
$0<q<2p/\gamma^2$.  Let $E_0$ be a fixed compact
subset of the domain of a $p$-dimensional volume or surface chaos.
Uniformly over sets $A\subset E_0$ satisfying
$
 \operatorname{diam}(A)\leq r\leq1
$
and over $0\leq\varepsilon\leq r$,
$$
 \E\big[M_{\gamma,\varepsilon}(A)^q\big]
 \lesssim_{\gamma,q,E_0}
 r^{\zeta_p(q)}.
$$
\end{prop}

See \cite[Theorems~2.11 and~2.14]{RhodesVargas2014} and
\cite[Propositions~2.3--2.4]{Bertacco2023}.  In
\cite{BerestyckiPowell}, Theorem~3.30 supplies the uniform moment bound for
the approximations, while Theorem~3.27, in particular equation~(3.55),
supplies the scaling relation.

\begin{proof}
The cited results give the estimate for Euclidean balls.  Any set of
diameter at most $r$ is contained in a ball of radius $r$, so positivity
gives the stated formulation. For surface chaos, cover $E_0$ by finitely
many relatively compact graph charts $\phi:V\subset\R^p\to S$. In each
chart,
$$
 c|u-v|\leq|\phi(u)-\phi(v)|\leq C|u-v|,
$$
and hence
$$
\begin{aligned}
 K(\phi(u),\phi(v))
 &=\log\frac1{|u-v|}+h(u,v),\\
 h(u,v)
 &=g(\phi(u),\phi(v))
 +\log\frac{|u-v|}{|\phi(u)-\phi(v)|}.
\end{aligned}
$$
The remainder $h$ is bounded by the distance comparison and compactness.
Moreover,
$$
 d\sigma_S(\phi(u))=J_\phi(u)\,du,
 \qquad
 J_\phi(u)=\sqrt{\det\big(D\phi(u)^{\mathsf T}D\phi(u)\big)}.
$$
This smooth density is bounded above and away from zero on each smaller
chart. Kahane comparison with a standard $p$-dimensional logarithmic
field therefore gives the moment estimate in each chart. Summing over
the finite cover completes the proof.
\end{proof}


The dimension formulas below concern volume chaos on $\T^p$, or on a
fixed open set compactly contained in the domain of the field in $\R^p$.
GMC is almost surely exact-dimensional, and in particular 
$$
 \dim_H(M_\gamma)
 =
 p-\frac{\gamma^2}{2}
 =
 \zeta_p'(1).
$$
Bertacco \cite{Bertacco2023} proved
the multifractal formalism of GMC and identified its almost sure $L^q$-spectrum.
For the torus, small flat coordinate patches preserve distances and
Lebesgue measure, so the covariance has the Euclidean form with a smooth
remainder. Exact dimensionality is local. A finite cover by such patches
also gives the same $L^q$-spectrum: the sums of powers of ball masses in
the patches and on the torus are comparable up to fixed factors and
fixed changes of radius.

With Bertacco's convention, write
$$
 \tau_{M_\gamma}(q)
 =
 \limsup_{r\to0}
 \frac{
 \log
 \sup
 \sum_i M_\gamma(B(x_i,r))^q
 }{-\log r},
$$
where the supremum is over disjoint balls of radius $r$ centered in the
support of the measure.  Its positive branch is
$$
 \tau_{M_\gamma}(q)
 =
 \begin{cases}
 p-\zeta_p(q),
 &0\leq q\leq q_{+,p},\\
 -q\alpha_{\gamma,p},
 &q\geq q_{+,p},
 \end{cases}
$$
where
$$
 q_{+,p}=\frac{\sqrt{2p}}{\gamma},
 \qquad
 \alpha_{\gamma,p}
 =
 \big(\sqrt p-\frac{\gamma}{\sqrt2}\big)^2.
$$

At $q=2$, the standard equivalence with correlation dimension gives,
on taking $p=d$ (see \cite[Section~2.1]{LinQiuTanI}),
$$
 \dim_2(M_\gamma)
 =
 -\tau_{M_\gamma}(2)
 =
 D_{\gamma,d},
$$
where
\begin{equation}
\label{eq:D-gamma-d}
 D_{\gamma,d}
 =
 \begin{cases}
 d-\gamma^2,
 &0<\gamma^2\leq d/2,\\[2mm]
 \big(\sqrt{2d}-\gamma\big)^2,
 &d/2<\gamma^2<2d.
 \end{cases}
\end{equation}
Since $D_{\gamma,d}<d$, \eqref{eq:reldim} also gives
\begin{equation}
\label{eq:dimbound}
 \dim_F(M_\gamma)
 \leq
 \dim_S(M_\gamma)
 =
 \dim_2(M_\gamma)
 =
 D_{\gamma,d}.
\end{equation}
Almost surely, the same correlation and Sobolev dimension identities hold
simultaneously for every set $A\Subset U$ with nonempty interior.
Indeed, take a countable basis of balls compactly contained in $U$ and
work on the probability-one event where the local formulas hold on every
such ball. They also hold on finite unions of these balls, by positivity
and Cauchy--Schwarz for Riesz energies. Given $A$, choose a basis ball
$B\subset A$ and a finite union $V$ of basis balls containing $\overline A$.
Then
$$
 C_{M_{\gamma,B}}(r)
 \leq C_{M_{\gamma,A}}(r)
 \leq C_{M_{\gamma,V}}(r).
$$
The two outside measures have correlation dimension $D_{\gamma,d}$,
so the middle one does too. Since $D_{\gamma,d}<d$, \eqref{eq:reldim}
also gives its Sobolev dimension.

We finally record the two variational identities used in the proofs.  A
direct optimization gives
\begin{equation}
\label{eq:D-variational}
 \frac{D_{\gamma,d}}2
 =
 \sup_{\substack{1<q\leq2\\q<2d/\gamma^2}}
 \frac{\zeta_d(q)-d}{q}.
\end{equation}
Moreover, whenever $\gamma^2<2p$,
$$
 \alpha_{\gamma,p}
 =
 \zeta_p'(q_{+,p})
 =
 \sup_{1<q<2p/\gamma^2}
 \frac{\zeta_p(q)-p}{q}.
$$

\subsection{Finite-range model}
\label{sec:finite-range-model}

Auxiliary fields with scaling properties are commonly used to establish
GMC moment estimates. Here we use the standard $\star$-scale-invariant
covariance construction \cite[Section~2.2]{RhodesVargas2014}, with a smooth
compactly supported kernel. This choice gives smooth coarse fields with
quantitative derivative bounds and exact independence between sufficiently
separated fine fields. The continuous scale parameter also provides the
martingales used in \cref{sec:curved-lower}. On the torus, we use the
periodization of this construction.

Choose a radial positive-definite function $k\in C_c^\infty(\R^d)$,
supported in $B(0,1)$ and normalized by $k(0)=1$. For example, take a
nonzero real-valued radial $\phi\in C_c^\infty(B(0,1/2))$ and set
$$
 k=\frac{\phi*\widetilde\phi}{(\phi*\widetilde\phi)(0)},
 \qquad \widetilde\phi(x)=\phi(-x).
$$
In Euclidean space set $k_t(z)=k(z/t)$, while on the torus set
$$
 k_t(z)=\sum_{\ell\in\Z^d}
 k\big(\frac{z+2\pi\ell}{t}\big),
 \qquad 0<t\leq1.
$$
Periodization preserves positive definiteness, since
$$
 \int_{\T^d}k_t(z)e^{-im\cdot z}\,dz
 =t^d\widehat k(tm)\geq0,
 \qquad m\in\Z^d.
$$
In both settings, $k_t(0)=1$ and $k_t(x-y)=0$ whenever
$\dist(x,y)>t$, where $\dist$ denotes Euclidean or torus distance.

For scale intervals $I\subset(0,1]$ bounded away from zero, let $X_I$ be
jointly defined centered Gaussian fields with covariance
$$
 \E\big[X_I(x)X_J(y)\big]
 =\int_{I\cap J}k_t(x-y)\,\frac{dt}{t}.
$$
Positive definiteness gives the existence of this family. The covariance
formula makes it additive over disjoint scale intervals, with independent
increments. We write
$$
 X_{a,b}=X_{(a,b]},
 \qquad X_\varepsilon=X_{\varepsilon,1},
 \qquad 0<a<b\leq1.
$$
For $a>0$, the field $x\mapsto X_{a,b}(x)$ has an almost surely smooth
version, by the smoothness of its covariance kernel.

The covariance of $X_\varepsilon$ is
$$
 K_\varepsilon(x,y)=\int_\varepsilon^1k_t(x-y)\,\frac{dt}{t}.
$$
Since $k(0)=1$ and $k$ is smooth and compactly supported,
$$
 K_\varepsilon(x,y)
 =\log_+\frac1{\dist(x,y)\vee\varepsilon}+O(1)
$$
uniformly in $x,y$ and $0<\varepsilon<1$. The limiting covariance is
smooth away from the diagonal and has a smooth remainder at the diagonal.
Indeed, write $k(z)=\kappa(|z|)$. For $0<|z|<1$,
$$
 K(z)=\int_{|z|}^1\kappa(u)\,\frac{du}{u}
 =\log\frac1{|z|}+h(z),
$$
where
$$
 h(z)=c_k-\int_0^1\frac{k(tz)-1}{t}\,dt,
 \qquad
 c_k=\int_0^1\frac{\kappa(u)-1}{u}\,du.
$$
Smooth radiality gives $\kappa(u)-1=O(u^2)$, so $c_k$ is finite.
Differentiation under the integral shows that $h$ is smooth. On the torus,
the same calculation applies near the diagonal, where only one periodic
translate contributes. Thus this field and any smooth log-correlated field
have a smooth covariance difference on compact subsets.

For $0<\delta<1$, set $X_{>\delta}=X_{\delta,1}$. Let $X$ and
$X_{\leq\delta}$ denote the limits of $X_{\varepsilon,1}$ and
$X_{\varepsilon,\delta}$, respectively, in the sense of distributions as
$\varepsilon\to 0$. Independence of disjoint scale intervals gives
$$
 X=X_{>\delta}+X_{\leq\delta},
$$
with independent summands. The coarse covariance satisfies
$$
 K_{>\delta}(x,y)=\int_\delta^1k_t(x-y)\,\frac{dt}{t},
 \qquad K_{>\delta}(x,x)=\log\frac1\delta.
$$
Moreover, $|\partial^\alpha k_t(z)|\lesssim_\alpha t^{-|\alpha|}$.
These estimates give the coarse derivative bounds proved in
\cref{lem:localmoments}. For $x\neq y$, the fine covariance is
$$
 K_{\leq\delta}(x,y)=\int_0^\delta k_t(x-y)\,\frac{dt}{t}.
$$
If $\dist(x,y)>\delta$, every integrand vanishes. Consequently, the
restrictions of the entire fine field, and of its chaos, to sets separated
by more than $\delta$ are independent. This is the finite-range property
used below.

We keep track of all the revealed scales by setting
$$
 \cF_{>\delta}=\sigma(X_{>r}:\delta\leq r\leq1),
$$
where $X_{>1}=0$. Put $Y_t=X_{>e^{-t}}$ and
$\cF_t=\cF_{>e^{-t}}$. When using stochastic calculus, we complete this
filtration with its null sets and make it right-continuous. The Gaussian increment estimates
give a version for which every spatial derivative
$\partial_x^\alpha Y_t(x)$ is jointly continuous in $(t,x)$ on
$[0,T]\times E$, for every finite $T$ and compact set $E$.
For each $x$, $Y_t(x)$ is a standard Brownian motion, and
$$
 d\langle Y(x),Y(y)\rangle_t=k_{e^{-t}}(x-y)\,dt.
$$
This is the continuous Gaussian martingale structure used below.

Define the coarse density and measure by
$$
 \rho_{>\delta}(x)
 =e^{\gamma X_{>\delta}(x)
 -\frac{\gamma^2}{2}\E[X_{>\delta}(x)^2]}, \qquad M_{>\delta}(dx)=\rho_{>\delta}(x)\,dx.
$$
Let $M_{\leq\delta}$ be the positive chaos of the fine field.
Independence of the scales gives
$$
 M_\gamma(dx)=\rho_{>\delta}(x)M_{\leq\delta}(dx).
$$
Subcritical convergence in $L^1$ preserves the normalization
$\E[M_{\leq\delta}(B)]=|B|$ for bounded sets $B$.
Since the fine field is independent of $\cF_{>\delta}$,
$$
 \E\big[M_\gamma(dx)\mid\cF_{>\delta}\big]=M_{>\delta}(dx).
$$
We therefore define the centered fine fluctuation by
$$
 \widetilde M_{\leq\delta}
 =M_\gamma-M_{>\delta}
 =\rho_{>\delta}\big(M_{\leq\delta}-dx\big),
 \qquad
 \E\big[\widetilde M_{\leq\delta}\mid\cF_{>\delta}\big]=0.
$$
The Fourier estimates below use the additive decomposition
$M_\gamma=M_{>\delta}+\widetilde M_{\leq\delta}$. The coarse term is
treated by integration by parts, while the centered fine term is estimated
conditionally on the coarse scales. The multiplicative representation
identifies these fine fluctuations and is also used in the lower bounds
for local masses. The same identities hold for surface chaos, with $dx$
replaced by surface measure, provided $\gamma^2<2p$ on a $p$-dimensional
surface.

The next proposition justifies proving the Fourier estimates for this
model.

\begin{prop}[Smooth transfer]
\label{prop:smooth-transfer}
Let $X$ and $Y$ be smooth log-correlated Gaussian fields either on $\T^d$
or on a common open set $U\subset\R^d$, and suppose that the difference of
their covariance kernels is smooth. On $\T^d$, the Fourier dimensions of
their subcritical GMC measures have the same law, and so do their Fourier
spectra. In the Euclidean setting, the same conclusions hold after
restricting both measures to any fixed set $A\Subset U$ of positive
Lebesgue measure. The same conclusions hold for chaos constructed with
surface measure on a fixed compact smooth $p$-dimensional submanifold of
$U$, provided $\gamma^2<2p$.
\end{prop}

\begin{proof}
We use the Gaussian completion construction in
\cite[Lemmas~3.1--3.2]{JunnilaSaksmanWebb2019}. (It is also straightforward to do by hand.) In the Euclidean setting,
we first multiply the covariance difference by $\psi(x)\psi(y)$, where
$\psi\in C_c^\infty(U)$ equals one near $\overline A$. That construction
writes the localized difference as the difference of two
positive-semidefinite kernels. It preserves every Sobolev order, so
these kernels are smooth under our assumptions.
On the torus the same argument uses Fourier series. In the surface case,
we construct the completion in an ambient neighbourhood before restricting
it to the surface.

Write $K_X-K_Y=H_+-H_-$ on the set under consideration, where
$H_+$ and $H_-$ are the smooth positive-semidefinite kernels
obtained above. Choose centered Gaussian fields $G_X$ and $G_Y$
with covariance kernels $H_-$ and $H_+$, independent respectively
of $X$ and $Y$. These fields have almost surely smooth versions.
Since $K_X+H_-=K_Y+H_+$, the completed fields have the same
covariance and hence
$$
 X+G_X\stackrel{(d)}{=}Y+G_Y.
$$
The corresponding chaos measures satisfy
$$
 M_\gamma^{X+G_X}(dx)
 =e^{\gamma G_X(x)-\frac{\gamma^2}{2}\E[G_X(x)^2]}
 M_\gamma^X(dx),
$$
and similarly for $Y$.

It remains to observe that multiplication by a smooth strictly positive
function preserves Fourier decay and the Fourier spectrum. Let $\mu$ be
compactly supported, let $f>0$ be smooth on a neighbourhood $W$ of
$\supp\mu$, and choose $\chi\in C_c^\infty(W)$ equal to one near
$\supp\mu$. Then
$$
 \widehat{f\mu}=(2\pi)^{-d}\widehat{\chi f}*\widehat\mu.
$$
For $\beta\in\R$, applying
the estimate in \eqref{eq:elemes} inside the convolution integral and then using
Young's inequality gives, for $1\leq p\leq\infty$,
$$
 \big\|(1+|\cdot|)^\beta\widehat{f\mu}\big\|_{L^p}
 \leq
 (2\pi)^{-d}
 \big\|(1+|\cdot|)^{|\beta|}\widehat{\chi f}\big\|_{L^1}
 \big\|(1+|\cdot|)^\beta\widehat\mu\big\|_{L^p}.
$$
The first norm on the right is finite because $\widehat{\chi f}$ is
rapidly decreasing. Applying the same argument with $1/f$ near
$\supp\mu$ gives the converse. The case $p=\infty$ preserves every
polynomial Fourier decay estimate, and the finite $p$ cases preserve the
Fourier spectrum energies for positive energy parameters, since bounded
frequencies contribute a finite amount. On the torus, the same proof uses
weighted discrete convolution. Applying this pathwise to the two random
multipliers proves the proposition.
\end{proof}

\begin{rmk}
The smoothness assumption in \cref{prop:smooth-transfer} is
deliberate. A covariance difference which is bounded on the compact set
under consideration is sufficient for two-sided Kahane comparisons and
hence for local moment estimates. Boundedness alone, however, gives no
transfer principle for pointwise Fourier decay. The latter uses the fact
that a smooth Gaussian completion multiplies the chaos by an almost surely
smooth positive density.
\end{rmk}


\subsection{Gaussian and probabilistic tools}

We collect the probabilistic estimates used below. The torus argument
uses only \cref{lem:conditional-type,lem:gaussian-shift}. For a continuous real
local martingale $M$, we denote by $\langle M\rangle$ its bracket, namely
the unique continuous increasing predictable process starting from zero
such that
$
 M_t^2-\langle M\rangle_t
$
is a local martingale.  If $Z=U+iV$ is a continuous complex local
martingale, we use the convention
$$
 \langle Z\rangle_t
 =
 \langle U\rangle_t+\langle V\rangle_t.
$$
We begin with standard martingale inequalities, see for instance
\cite[Chapters~II and~IV]{RevuzYor}.

\begin{lem}[Martingale inequalities]
\label{lem:martingale-inequalities}
Let $T$ be a bounded stopping time.  The following estimates hold.
\begin{enumerate}[label=\textup{(\roman*)}]
 \item If $(Y_t)_{t\geq0}$ is a nonnegative continuous
 submartingale, $q>1$ and $Y_T\in L^q$, then
 $$
  \E\big[\sup_{t\leq T}Y_t^q\big]
  \leq
  \big(\frac{q}{q-1}\big)^q\E[Y_T^q].
 $$

 \item Let $Z=U+iV$ be a continuous complex local martingale with
 $Z_0=0$.  For every $u,v>0$,
 $$
  \PP(
  |Z_T|>u,\ \langle Z\rangle_T\leq v
  )
  \leq
  4e^{-\frac{u^2}{4v}}.
 $$

 \item Under the assumptions of \textup{(ii)}, for every $q>0$,
 $$
  \E\big[
  |Z_T|^q\1_{\{\langle Z\rangle_T\leq v\}}
  \big]
  \lesssim_q
  v^{q/2}.
 $$
\end{enumerate}
\end{lem}

\begin{proof}
Part \textup{(i)} is Doob's $L^q$ maximal inequality.  

For \textup{(ii)}, let $M$ be a continuous real local martingale
with $M_0=0$.  For $\lambda>0$, the exponential process
$$
 \mathcal E_t
 = e^{\lambda M_t-\frac{\lambda^2}{2}\langle M\rangle_t}
 $$
is a nonnegative local martingale and hence a supermartingale. Therefore
$$
 \PP\big(M_T>a,\ \langle M\rangle_T\leq v\big)
 \leq
 \PP\big(
 \mathcal E_T>
 e^{\lambda a-\lambda^2v/2}
 \big)
 \leq e^{-\lambda a+\lambda^2v/2}
$$
Taking $\lambda=a/v$ and applying the same argument to $-M$ gives
\begin{equation}
\label{eq:real-martingale-tail}
 \PP\big(
 |M_T|>a,\ \langle M\rangle_T\leq v
 \big)
 \leq
 2e^{-a^2/(2v)}.
\end{equation}
If $|Z_T|>u$, then either $|U_T|>u/\sqrt2$ or
$|V_T|>u/\sqrt2$.  Moreover,
$
 \langle Z\rangle_T\leq v
$
implies that both real brackets are at most $v$.  A union bound and
\eqref{eq:real-martingale-tail} prove \textup{(ii)}.

Finally, the layer-cake formula and \textup{(ii)} give
$$
\begin{aligned}
 \E\big[
 |Z_T|^q\1_{\{\langle Z\rangle_T\leq v\}}
 \big]
 =
 q\int_0^\infty
 u^{q-1}
 \PP\big(
 |Z_T|>u,\ \langle Z\rangle_T\leq v
 \big)\,du
 \lesssim_q
 v^{q/2},
\end{aligned}
$$
which proves \textup{(iii)}.
\end{proof}

In the proof of \cref{lem:curved-upper-boundary-maximum}, we also use
the following standard consequence of Dudley's entropy inequality
(see for instance \cite[Section~1.3]{AdlerTaylor}).  For a compact pseudometric space $(T,d)$
and $\varepsilon>0$, let $\mathcal N(T,d,\varepsilon)$ denote the smallest
number of $d$-balls of radius $\varepsilon$ needed to cover $T$.

\begin{lem}[Dudley's entropy inequality]
\label{lem:dudley}
Let $(G(x))_{x\in T}$ be a centered separable Gaussian process with
canonical pseudometric
$$
 d_G(x,y)
 =
 (\E[|G(x)-G(y)|^2])^{1/2}.
$$
Set
$
 A^2=\sup_{x\in T}\Var(G(x)).
$
Then
$$
 \E\big[\sup_{x\in T}|G(x)|\big]
 \lesssim
 A+
 \int_0^{2A}
 \sqrt{\log\mathcal N(T,d_G,\varepsilon)}\,d\varepsilon.
$$
In particular, suppose that $T$ is a fixed compact subset of $\R^m$ and,
for some $A,L>0$,
$$
 \sup_{x\in T}\Var(G(x))\leq A^2,
 \qquad
 d_G(x,y)\leq L|x-y|.
$$
Then
$$
 \E\big[\|G\|_{L^\infty(T)}\big]
 \lesssim_{m,T}
 A\sqrt{\log\big(2+\frac{L}{A}\big)}.
$$
\end{lem}

\begin{proof}
Fix $x_0\in T$.  The first estimate is Dudley's entropy inequality applied
to the increment process $G(x)-G(x_0)$ and to its negative, together with
$
 \E[|G(x_0)|]\lesssim A
$
and
$
 \operatorname{diam}(T,d_G)\leq2A.
$

For the second estimate, every Euclidean ball of radius $\varepsilon/L$ is
contained in a ball of radius $\varepsilon$ for the pseudometric $d_G$.
Hence a Euclidean $\varepsilon/L$-net of $T$ is a $d_G$-net, and
$$
 \mathcal N(T,d_G,\varepsilon)
 \lesssim_{m,T}
 \big(1+\frac{L}{\varepsilon}\big)^m.
$$
Substitution into the entropy integral gives the result.
\end{proof}

The next lemma is a special case of the martingale type inequalities for $L^p$ spaces (see
\cite[Section~3.5.d]{HytonenVanNeervenVeraarWeis2016}). Weighted $L^{2k}$ and $\ell^{2k}$ spaces have
type $2$, and hence type $\theta$ for every $1<\theta\leq2$. The general
martingale type inequality is used in the works \cite{ChenLinQiu}
and
\cite{LinQiuTanII} on Fourier dimension. In the conditionally independent case needed here,
symmetrization and Khintchine's inequality give a short proof.

\begin{lem}
\label{lem:conditional-type}
Let $1<\theta\leq2$, let $k$ be a positive integer, and let $B$ be a
weighted $L^{2k}$ or $\ell^{2k}$ space.  Let $(Z_j)_{j\in J}$ be a finite
family in $L^\theta(\Omega,B)$ whose members, conditionally on a $\sigma$-field
$\mathcal G$, are independent and satisfy
$
 \E[Z_j\mid\mathcal G]=0.
$
Then
$$
 \E\Big[
 \big\|\sum_{j\in J}Z_j\big\|_B^\theta
 \,\Big|\,\mathcal G
 \Big]
 \lesssim_{\theta,k}
 \sum_{j\in J}\E[\|Z_j\|_B^\theta\mid\mathcal G].
$$
\end{lem}

\begin{proof}
The weight can be incorporated into the underlying measure, so it is enough
to take $B=L^{2k}$. Condition on $\mathcal G$ and let $(Z_j')_{j\in J}$ be
an independent copy of $(Z_j)_{j\in J}$ under the conditional law. Let
$(\varepsilon_j)_{j\in J}$ be independent Rademacher variables.
Jensen's inequality, symmetry and the triangle inequality give
\begin{align*}
 \E\Big[
 \big\|\sum_jZ_j\big\|_B^\theta
 \,\big|\,\mathcal G
 \Big]
 &\leq
 \E\Big[
 \big\|\sum_j(Z_j-Z_j')\big\|_B^\theta
 \,\big|\,\mathcal G
 \Big]\\
 &\lesssim_\theta
 \E\Big[
 \E_\varepsilon
 \big\|\sum_j\varepsilon_jZ_j\big\|_{2k}^\theta
 \,\big|\,\mathcal G
 \Big].
\end{align*}
For every realization of $(Z_j)$, the Khintchine and Minkowski inequalities
give
\begin{align*}
 \E_\varepsilon
 \big\|\sum_j\varepsilon_jZ_j\big\|_{2k}^\theta
 &\leq
 \big(
 \E_\varepsilon
 \big\|\sum_j\varepsilon_jZ_j\big\|_{2k}^{2k}
 \big)^{\theta/(2k)}\\
 &\lesssim_k
 \big\|\big(\sum_j|Z_j|^2\big)^{1/2}\big\|_{2k}^\theta\\
 &\leq
 \big(\sum_j\|Z_j\|_{2k}^2\big)^{\theta/2}
 \leq
 \sum_j\|Z_j\|_{2k}^\theta.
\end{align*}
Taking the remaining conditional expectation proves the claim.
\end{proof}

We use Kahane's convexity inequality in the following standard form,
see \cite[Theorem~3.19]{BerestyckiPowell}.

\begin{prop}[Kahane's convexity inequality]
\label{prop:kahane}
Let $T$ be a compact metric space, let $\sigma$ be a nonzero finite
measure on $T$, and let $X$ and $Y$ be continuous centered Gaussian fields
on $T$.  Suppose that
$$
 \E[X(x)X(y)]
 \leq
 \E[Y(x)Y(y)]
$$
for every $x,y\in T$.  For $\gamma\in\R$, set
$$
 Z_X
 =
 \int_T e^{\gamma X(x)-\frac{\gamma^2}{2}\E[X(x)^2]}\,\sigma(dx),
$$
and define $Z_Y$ analogously.  If $F:(0,\infty)\to\R$ is convex and
satisfies
$$
 |F(t)|
 \leq
 C(1+t^a+t^{-b}),
 \qquad t>0,
$$
for some $C,a,b\geq0$, then
$$
 \E[F(Z_X)]
 \leq
 \E[F(Z_Y)].
$$
If $F$ is concave and satisfies the same growth condition, the inequality
is reversed.
\end{prop}

In particular, if two covariance kernels differ by at most a fixed
constant in absolute value, adding independent Gaussian constants and
applying \cref{prop:kahane} in both directions gives two-sided
comparison of their moments, with constants depending only on the
covariance bound, the moment exponent and $\gamma$.

We also use the following Gaussian shift identity, sometimes called the
Cameron--Martin formula or Girsanov's lemma. In particular, this corresponds to a tilting of the last coordinate from the more general statement in
\cite[Lemma~2.5]{BerestyckiPowell}.

\begin{lem}[Gaussian shift identity]
\label{lem:gaussian-shift}
Let $G$ be a centered Gaussian vector and let $Z$ be a centered Gaussian
variable jointly Gaussian with $G$. For every nonnegative function
$F$,
$$
 \E\big[
 e^{Z-\tfrac12\Var Z}F(G)\big]
 =
 \E[F\big(G+\Cov(G,Z)\big)].
$$
\end{lem}

\section{The torus} \label{sec:3}

\subsection{A coarse derivative estimate}
For $0<\delta\leq 1$ and an integer $J\geq0$, define
$$
 N_{\delta,J}
 =
 \sum_{|\alpha|\leq J}\delta^{|\alpha|}
 \int_{\T^d}|\partial^\alpha\rho_{>\delta}(x)|\,dx.
$$
In the Euclidean setting, for a bounded set $E\Subset U$, set
$$
 N_{\delta,J}^E
 =
 \sum_{|\alpha|\leq J}\delta^{|\alpha|}
 \int_E|\partial^\alpha\rho_{>\delta}(x)|\,dx.
$$

\begin{lem}\label{lem:localmoments}
For every $0<\delta\leq1$, $u\geq1$, and integer $J\geq0$,
$$
 \E[N_{\delta,J}^u]
 \lesssim_{d,u,J}
 \delta^{-\frac{\gamma^2}{2}(u^2-u)}.
$$
The same holds for $N_{\delta,J}^E$, for any bounded $E\Subset U$.
\end{lem}

\begin{proof}
Since $k\in C_c^\infty(\R^d)$, the kernel and its periodization satisfy
$$
 |\partial^\alpha k_t(x)|
 \lesssim_{\alpha,d}t^{-|\alpha|}
$$
uniformly in $0<t\leq1$ and $x$. Hence, whenever
$|\alpha|+|\beta|\geq1$,
\begin{equation}
\label{eq:dercovest}
 |\partial_x^\alpha\partial_y^\beta K_{>\delta}(x,y)|
 \lesssim_{\alpha,\beta,d}
 \int_\delta^1t^{-|\alpha|-|\beta|}\frac{dt}{t}
 \lesssim_{\alpha,\beta,d}
 \delta^{-|\alpha|-|\beta|}.
\end{equation}
Moreover, $k_t(0)=1$ in both settings, so
$
 K_{>\delta}(x,x)=\log1/\delta
$
is independent of $x$. In the Euclidean setting, the auxiliary field is
defined on all of $\R^d$, and only the integral is restricted to $E$.

Since the variance is independent of $x$, the Faà di Bruno formula
shows that, for $1\leq|\alpha|\leq J$, the normalized derivative
$\delta^{|\alpha|}\partial^\alpha\rho_{>\delta}(x)$ is a finite linear
combination of terms
$$
 \rho_{>\delta}(x)
 \prod_{\ell=1}^r
 \delta^{|\beta_\ell|}\partial^{\beta_\ell}X_{>\delta}(x),
 \qquad
 \beta_1+\cdots+\beta_r=\alpha,
 \quad |\beta_\ell|\geq1.
$$
Fix $x$ and put
$$
 \mathbf D
 =
 \big(\delta^{|\beta|}\partial^\beta X_{>\delta}(x)\big)_{1\leq|\beta|\leq J}.
$$
The covariance matrix of $\mathbf D$ and the vector
$\Cov(\mathbf D,X_{>\delta}(x))$ are uniformly bounded by
\eqref{eq:dercovest}. \cref{lem:gaussian-shift} with $G=\mathbf D$ and
$Z=u\gamma X_{>\delta}(x)$ gives
$$
 \E\big[\rho_{>\delta}(x)^uF(\mathbf D)\big]
 =
 \delta^{-\frac{\gamma^2}{2}(u^2-u)}
 \E\big[F\big(\mathbf D+u\gamma\Cov(\mathbf D,X_{>\delta}(x))\big)\big].
$$
The expectation on the right is uniformly bounded when $F$ is the
$u$-th absolute power of any product above. Summing the finitely many terms, and taking $F\equiv1$ when
$\alpha=0$, yields
$$
 \delta^{u|\alpha|}
 \E\big[|\partial^\alpha\rho_{>\delta}(x)|^u\big]
 \lesssim_{d,u,J}
 \delta^{-\frac{\gamma^2}{2}(u^2-u)},
 \qquad 0\leq |\alpha|\leq J.
$$
Finally, applying Jensen's inequality to the finite sum over $\alpha$
and then to the integral over $E$ gives
$$
 \E\big[(N_{\delta,J}^E)^u\big]
 \lesssim_{u,J,E}
 \sum_{|\alpha|\leq J}\delta^{u|\alpha|}
 \int_E\E\big[|\partial^\alpha\rho_{>\delta}(x)|^u\big]\,dx
 \lesssim_{u,E,J}
 \delta^{-\frac{\gamma^2}{2}(u^2-u)}.
$$
The same calculation on $\T^d$ proves the torus estimate.
\end{proof}

\subsection{Proof of \texorpdfstring{\cref{thm:torus}}{Theorem~\getrefnumber{thm:torus}}}

We first treat the finite-range model.  Fix $0<s<D_{\gamma,d}$.  By
\eqref{eq:D-variational}, we may choose $1<q\leq2$ such that
$q<2d/\gamma^2$ and
$
 \zeta_d(q)-d>sq/2.
$
Choose $\eta>0$ sufficiently small that
\begin{equation}
 (1-\eta)(\zeta_d(q)-d)>\frac{sq}{2},
 \label{eq:choice-eta-torus}
\end{equation}
and put
$
 \delta_n=2^{-\lfloor(1-\eta)n\rfloor}.
$
Fix an integer $k\geq1$. 

\smallskip
\noindent\emph{The coarse part.}
For $m\neq0$, choose a coordinate $j=j(m)$ such that
$|m_j|\geq|m|/\sqrt d$. Integrating by parts $J$ times in $x_j$ gives
$$
 |\widehat{M_{>\delta}}(m)|
 \lesssim_{d,J}
 (\delta|m|)^{-J}N_{\delta,J}.
$$
Since $\delta_n|m|\geq2^{\eta n}$ for $m\in A_n^\T$ and
$\sum_{m\in A_n^\T}|m|^{ks-d}\asymp2^{nks}$,
\cref{lem:localmoments} at $u=1$ yields
$$
 \E\big[\|\widehat{M_{>\delta_n}}\|_{B_{k,s,n}}\big]
 \lesssim_{J,k}
 2^{-n(J\eta-s/2)}.
$$
Choosing $J\eta>s/2$ makes these first moments summable.

\smallskip
\noindent\emph{The fine part.}
The additive and multiplicative decompositions give
$$
 \widetilde M_{\leq\delta}(dx)
 =
 \rho_{>\delta}(x)(M_{\leq\delta}(dx)-dx).
$$
Partition $\T^d$ into $O(\delta^{-d})$ cubes $Q$ of side comparable to
$\delta=\delta_n$. Join two cubes when their torus distance is at most
$\delta$. Each cube has a bounded number
of neighbors, independently of $\delta$, so greedy coloring partitions
them into a fixed number of classes. Distinct cubes in each class are
separated by more than $\delta$. For $m\in A_n^\T$, set
$$
 Z_Q(m)
 =
 \int_Q e^{-im\cdot x}\,\widetilde M_{\leq\delta}(dx).
$$
Conditionally on $\cF_{>\delta}$, the coarse density is fixed. 
Thus the $B_{k,s,n}$-valued variables
$Z_Q$ in each class are conditionally independent and centered. Applying
\cref{lem:conditional-type} to each class and summing over the classes
gives
$$
 \E\big[
 \|\widehat{\widetilde M_{\leq\delta}}\|_{B_{k,s,n}}^q
 \mid\cF_{>\delta}
 \big]
 \lesssim_{k,q}
 \sum_Q
 \E[\|Z_Q\|_{B_{k,s,n}}^q\mid\cF_{>\delta}].
$$
Moreover,
$$
\|Z_Q\|_{B_{k,s,n}}
 \lesssim
 2^{ns/2}
 (M_\gamma(Q)
 +\E[M_\gamma(Q)\mid\cF_{>\delta}]).
$$
Taking expectations, using conditional Jensen, and then
\cref{prop:gmcmoments}, we obtain
\begin{align*}
 \E[
 \|\widehat{\widetilde M_{\leq\delta_n}}\|_{B_{k,s,n}}^q]
 &\lesssim_{q,k}
 2^{nsq/2}
 \sum_Q\E[M_\gamma(Q)^q]\\
 &\lesssim
 2^{nsq/2}\delta_n^{-d}
 \delta_n^{\zeta_d(q)}\\
 &\lesssim
 2^{-n((1-\eta)(\zeta_d(q)-d)-sq/2)}.
\end{align*}
By \eqref{eq:choice-eta-torus}, the last bound decays geometrically.
Jensen's inequality therefore gives summable first moments of the fine
annular norms. 

Combining the two parts gives
$$
 \E\Big[\sum_{n\geq0}
 \|\widehat M_\gamma\|_{B_{k,s,n}}\Big]<\infty.
$$
Thus these norms form an $\ell^1$ sequence almost surely, and hence also
an $\ell^{2k}$ sequence.
Taking the countable intersection over $k\geq1$ and applying
\cref{lem:weighted-fourier-criterion} gives
$\dim_F(M_\gamma)\geq s$.  Finally, take $s$ along a countable sequence
increasing to $D_{\gamma,d}$. The reverse bound follows from
\eqref{eq:dimbound}, so $\dim_F(M_\gamma)=D_{\gamma,d}$ for the finite-range model.
\cref{prop:smooth-transfer} gives the same conclusion for the
original smooth log-correlated field, which completes the proof.

\subsection{A dyadic finite-range criterion}
The proof also applies to independent smooth dyadic blocks whose
derivative estimates have polynomial losses in the block index. The
following proposition records this extension and includes the model of
Chen, Lin and Qiu without requiring a smooth covariance remainder.

\begin{prop}[Dyadic finite-range fields]
\label{prop:dyadic-finite-range}
Let $(\psi_j)_{j\geq1}$ be independent centered Gaussian fields on $\T^d$
with almost surely smooth sample paths, and write
$$
 K_j(x,y)=\E[\psi_j(x)\psi_j(y)].
$$
Let $X=\sum_{j\geq1}\psi_j$ be the associated generalized Gaussian field.
Assume that the following properties hold.
\begin{enumerate}[label=\textup{(\roman*)}]
 \item There is $C>0$ such that $K_j(x,y)=0$ whenever
 $d_{\T^d}(x,y)>C2^{-j}$.
 \item For $x\neq y$,
 $$
  \sum_{j\geq1}K_j(x,y)
  =
  \log_+\frac1{d_{\T^d}(x,y)}+g(x,y),
 $$
 where $g$ is bounded and continuous.
 \item For every nonzero multi-index $\alpha$, there are constants
 $A_\alpha, C_\alpha>0$ such that
 $$
  \sup_{x\in\T^d}
  \E\big[|\partial^\alpha\psi_j(x)|^2\big]
  \leq
  C_\alpha j^{A_\alpha}2^{2j|\alpha|}.
 $$
\end{enumerate}
For $0<\gamma<\sqrt{2d}$, let $M_\gamma$ be the subcritical GMC of $X$.
Then, almost surely,
$$
 \dim_F(M_\gamma)=D_{\gamma,d}.
$$
\end{prop}

\begin{proof}
Put $Y_N=\sum_{j=1}^N\psi_j$, $V_N(x)=\E[Y_N(x)^2]$, and
$
 \rho_N(x)=e^{\gamma Y_N(x)-\frac{\gamma^2}{2}V_N(x)}.
$
The remaining field $X-Y_N$ is independent of the first $N$ blocks and is
log-correlated with a bounded continuous remainder. The mean of its subcritical chaos is
Lebesgue measure, so the multiplicative decomposition gives
$$
 \E[M_\gamma(dx)\mid\cF_N]=\rho_N(x)\,dx,
 \qquad \cF_N=\sigma(\psi_1,\ldots,\psi_N).
$$
In particular, these conditional expectations form the uniformly
integrable block martingale converging to $M_\gamma$.

Fix $0<s<D_{\gamma,d}$ and choose $q$ and $\eta$ as in the proof of
\cref{thm:torus}. Fix also an integer $k\geq1$.
By \textup{(i)}, the field formed from the blocks with $j>N$ has dependence
range $O(2^{-N})$. The cube coloring and conditional independent-sum
estimate used above therefore give
$$
 \E\big[
 \|\widehat{M_\gamma-\rho_N\,dx}\|_{B_{k,s,n}}^q
 \big]
 \lesssim_{k,q}
 2^{nsq/2}\sum_Q\E[M_\gamma(Q)^q]
 \lesssim
 2^{nsq/2}2^{-N(\zeta_d(q)-d)},
$$
where the cubes have side comparable to $2^{-N}$ and the last step follows
from \textup{(ii)} and \cref{prop:gmcmoments}.

For the coarse density, first note that \textup{(i)} and \textup{(iii)}
give a polynomial bound on the block variances. For all sufficiently
large $j$, choose $y$ with
$d_{\T^d}(x,y)=(C+1)2^{-j}$. Since $K_j(x,y)=0$,
$$
 K_j(x,x)
 \leq \E\big[|\psi_j(x)-\psi_j(y)|^2\big]
 \lesssim 2^{-2j}\sup_z\E\big[|\nabla\psi_j(z)|^2\big]
 \lesssim (1+j)^A
$$
for some $A>0$. Absorbing the finitely many remaining blocks gives
$$
 \sup_x V_N(x)\lesssim(1+N)^{A+1}.
$$
Independence and \textup{(iii)} also imply, for every fixed $J$ and
$1\leq|\alpha|\leq J$,
$$
 2^{-2N|\alpha|}\E\big[|\partial^\alpha Y_N(x)|^2\big]
 \lesssim_J(1+N)^{A_J}.
$$
Cauchy--Schwarz therefore bounds the covariances of these normalized
derivatives with $Y_N(x)$ by a polynomial in $N$.
The same holds for $2^{-N|\alpha|}\partial^\alpha V_N(x)$, since
derivatives of $V_N$ are sums of covariances of derivatives of $Y_N$.
The Faà di Bruno and Gaussian shift calculation in
\cref{lem:localmoments}, at the first-moment level, gives
$$
 \E\Big[
 \sum_{|\alpha|\leq J}2^{-N|\alpha|}
 \int_{\T^d}|\partial^\alpha\rho_N(x)|\,dx
 \Big]
 \lesssim_J(1+N)^{C_J}.
$$
The variance normalization cancels the exponential moment at this order. Integration by parts and Jensen's inequality for the fine part
therefore give
$$
 \E\big[\|\widehat M_\gamma\|_{B_{k,s,n}}\big]
 \lesssim_{k,q,J}2^{ns/2}(
 2^{-N(\zeta_d(q)-d)/q}
 +(1+N)^{C_J}2^{-J(n-N)}).
$$
Take $N=\lfloor(1-\eta)n\rfloor$ and $J\eta>s/2$. Both terms are
summable. Thus the annular norms are almost surely summable, and so are
their $2k$-th powers.
Taking the countable intersection over $k$ and applying
\cref{lem:weighted-fourier-criterion} gives $\dim_F(M_\gamma)\geq s$.
Finally, take $s$ along a countable sequence increasing to $D_{\gamma,d}$.
On the other hand, \textup{(ii)} gives the usual correlation dimension formula
$\dim_2(M_\gamma)=D_{\gamma,d}<d$, so \eqref{eq:reldim} gives the reverse
bound.
\end{proof}

\begin{rmk}[The model of Chen, Lin and Qiu]
After rescaling their torus to $(\R/2\pi\Z)^d$, the field constructed in
\cite[Proposition~3.1]{ChenLinQiu} satisfies the preceding proposition.
Their properties \textup{(P0)}--\textup{(P2)} give independence, smoothness,
finite dependence and the logarithmic covariance.  To check \textup{(iii)} for every fixed
multi-index, use their construction $\psi_j=P_j*\xi_j$, where $P_j$ is a
smooth mollifier at scale $\varepsilon_j=j^{-2}2^{-j}$ and
$\E[\xi_j(x)^2]=\log2$. Then
$$
 \sup_x\|\partial^\alpha\psi_j(x)\|_{L^2(\Omega)}
 \leq\sqrt{\log2}\,\|\partial^\alpha P_j\|_{L^1}
 \lesssim_\alpha j^{2|\alpha|}2^{j|\alpha|}.
$$
Squaring gives \textup{(iii)}, including orders beyond the range
$|\alpha|\leq d$ stated in their property \textup{(P4)}. Thus the
proposition recovers their Theorem~1.1 directly from the block construction.
\end{rmk}

\section{Euclidean restrictions and flat boundaries}
\label{sec:4}
Let $U\subset\R^d$ be open and fix $0<\gamma<\sqrt{2d}$.
We perform the finite-range reduction in a
fixed bounded open set $W\Subset U$ containing the supports of the
measures under consideration. For a set $A\Subset W$, write
$$
 M_{>\delta,A}=\1_A M_{>\delta},
 \qquad
 \widetilde M_{\leq\delta,A}=\1_A\widetilde M_{\leq\delta}.
$$
The following estimate concerns the finite-range field and holds for every
bounded deterministic multiplier, independently of its boundary behavior.

\begin{lem}[Fine fluctuation]\label{lem:domain-fine-part}
Let $f$ be a bounded function supported in $W$. For every
${1<\theta\leq2}$ with $\theta<2d/\gamma^2$, every $s\geq0$, every
integer $k\geq1$ and $0<\delta\leq1$,
$$
 \E\Big[
 \|\widehat{f\widetilde M_{\leq\delta}}\|_{B_{k,s,n}^\R}^{\theta}
 \Big]
 \lesssim_{\theta,k,s,W}
 \|f\|_\infty^\theta
 2^{ns\theta/2}\delta^{\zeta_d(\theta)-d}.
$$
In particular, if $s<D_{\gamma,d}$, there is $\theta\in(1,2]$ such that, for every sufficiently small $\eta>0$ and
$\delta_n=2^{-\lfloor(1-\eta)n\rfloor}$, the left-hand side is
$O_{\eta,f,k}(2^{-cn})$ for some $c>0$.
\end{lem}

\begin{proof}
Partition $W$ into intersections $Q\cap W$, where the cubes $Q$ have side comparable to $\delta$,
and set
$$
 Z_Q(\xi)=\int_{Q\cap W}e^{-i\xi\cdot x}f(x)
 \widetilde M_{\leq\delta}(dx).
$$
There are $O_W(\delta^{-d})$ cubes. The bounded-degree coloring used in
the torus proof partitions them into a fixed number of classes within
which the $Z_Q$ are conditionally independent and centered. Moreover,
$$
 \|Z_Q\|_{B_{k,s,n}^\R}
 \lesssim_{k,s} 2^{ns/2}\|f\|_\infty
 \big(M_\gamma(Q\cap W)+M_{>\delta}(Q\cap W)\big).
$$
\cref{lem:conditional-type} and conditional Jensen's inequality,
using $M_{>\delta}(Q\cap W)=\E[M_\gamma(Q\cap W)\mid\cF_{>\delta}]$,
give
$$
 \E\Big[
 \|\widehat{f\widetilde M_{\leq\delta}}\|_{B_{k,s,n}^\R}^{\theta}
 \Big]
 \lesssim_{k,s} 2^{ns\theta/2}\|f\|_\infty^\theta
 \sum_Q\E\big[M_\gamma(Q\cap W)^\theta\big].
$$
\cref{prop:gmcmoments} proves the first assertion. For the
second, use \eqref{eq:D-variational} to choose $\theta$ and then
$\eta_1>0$ such that
$$
 (1-\eta_1)\big(\zeta_d(\theta)-d\big)>s\theta/2
$$
to conclude.
\end{proof}

\begin{lem}[Smooth approximation]\label{lem:multiplier-approximation}
Let $M_\gamma$ be associated with a smooth log-correlated field on $U$,
and let $f\geq0$ be a bounded function compactly supported in $W$.
Suppose that, for some $\sigma>0$ and all sufficiently small $\delta$,
there are $a_\delta\in C_c^\infty(W)$ such that
$$
 \|f-a_\delta\|_1\lesssim\delta^\sigma,
 \qquad
 \|\partial^\alpha a_\delta\|_\infty
 \lesssim_\alpha\delta^{-|\alpha|}
$$
for every multi-index $\alpha$. Then, almost surely,
$$
 \dim_F(fM_\gamma)\geq\min\{2\sigma,D_{\gamma,d}\}.
$$
\end{lem}

\begin{proof}
The proof of \cref{prop:smooth-transfer} applies with the
factor $f$ inserted, so it is enough to consider the finite-range model.
Write
$$
 fM_\gamma
 =a_\delta\rho_{>\delta}\,dx
 +(f-a_\delta)\rho_{>\delta}\,dx
 +f\widetilde M_{\leq\delta}.
$$
We bound the Fourier transform of the middle term uniformly by its $L^1$ norm. The expectation of this norm is
$\|f-a_\delta\|_1\lesssim\delta^\sigma$. For the first term, Leibniz's
rule and $J$ integrations by parts give
$$
 |\widehat{a_\delta\rho_{>\delta}}(\xi)|
 \lesssim_J (\delta|\xi|)^{-J}N_{\delta,J}^W.
$$
Since $\E[N_{\delta,J}^W]\lesssim_J1$ by
\cref{lem:localmoments}, \cref{lem:domain-fine-part} and
Jensen's inequality yield
$$
 \E\big[\|\widehat{fM_\gamma}\|_{B_{k,s,n}^\R}\big]
 \lesssim_{\theta,f,k,s,J,W}2^{ns/2}\big(
 \delta^\sigma+(2^n\delta)^{-J}
 +\delta^{(\zeta_d(\theta)-d)/\theta}\big).
$$
Fix $0<s<\min\{2\sigma,D_{\gamma,d}\}$, choose $\theta$ and a
sufficiently small $\eta>0$ such that
$$
 (1-\eta)\min\Big\{\sigma,
 \frac{\zeta_d(\theta)-d}{\theta}\Big\}>s/2,
$$
and take $\delta=\delta_n=2^{-\lfloor(1-\eta)n\rfloor}$ and
$J\eta>s/2$. The displayed expectations are summable. Thus the annular
norms, and hence their $2k$-th powers, are almost surely summable.
As before, take the countable intersection over $k$, apply
\cref{lem:weighted-fourier-criterion}, and let $s$ increase to
$\min\{2\sigma,D_{\gamma,d}\}$ along a countable sequence.
\end{proof}

\subsection{Proof of \texorpdfstring{\cref{thm:finite-perimeter}}{Theorem~\getrefnumber{thm:finite-perimeter}}} \label{sec:prooffiniteper}

Let $\varphi$ be a nonnegative smooth mollifier supported in the unit
ball and set $a_\delta=\1_A*\varphi_\delta$. The finite perimeter
assumption means that the distributional gradient $D\1_A$ is a finite
measure of total variation $\operatorname{Per}(A)$. The standard
translation bound
$$
 \|\1_A(\,\cdot-h)-\1_A\|_1
 \leq |h|\operatorname{Per}(A)
$$
follows first for smooth functions by integrating their gradient along
line segments, and then for $\1_A$ by mollification. Consequently,
\begin{equation}\label{eq:domain-bv-approximation}
 \|\1_A-a_\delta\|_1
 \leq\operatorname{Per}(A)\int|h|\varphi_\delta(h)\,dh
 \lesssim_A\delta.
\end{equation}
The mollifier bounds also give
$\|\partial^\alpha a_\delta\|_\infty\lesssim_\alpha\delta^{-|\alpha|}$,
and $a_\delta\in C_c^\infty(W)$ for small $\delta$. Applying
\cref{lem:multiplier-approximation} with $f=\1_A$ and $\sigma=1$ finishes the proof.

\begin{rmk}[A fractional extension]
\label{rmk:fractional-boundary}
The proof uses finite perimeter only through
\eqref{eq:domain-bv-approximation}. If, more generally,
$$
 \|\1_A-\1_A*\varphi_\delta\|_1\lesssim_A\delta^\sigma
$$
for some $0<\sigma\leq1$, \cref{lem:multiplier-approximation}
gives $\dim_F(M_{\gamma,A})\geq\min\{2\sigma,D_{\gamma,d}\}$ almost
surely.
\end{rmk}

We can also use \cref{lem:multiplier-approximation} to give an easy proof of \cref{cor:smooth-localization}.

\begin{proof}[Proof of \cref{cor:smooth-localization}]
Apply \cref{lem:multiplier-approximation} with
$f=a_\delta=\chi$ and arbitrary $\sigma>0$ to obtain the lower bound
$D_{\gamma,d}$. For the reverse bound, choose a ball $B\Subset U$ on
which $\chi\geq c>0$. Since $\chi M_\gamma\geq cM_{\gamma,B}$ and
$\dim_2(M_{\gamma,B})=D_{\gamma,d}<d$, comparison of their Riesz
energies gives $\dim_S(\chi M_\gamma)\leq D_{\gamma,d}$. Now use the bound $\dim_F \leq \dim_S$ to conclude.
\end{proof}

\subsection{Sharpness at flat boundary patches}

We first record a deterministic bound which will also be used in \cref{sec:5}.

\begin{lem}[A one-sided mass bound]\label{lem:one-sided-mass}
Let $\mu$ be a finite positive measure on $\R^d$ such that
$|\widehat\mu(\xi)|\lesssim(1+|\xi|)^{-s/2}$ for some $s>0$.
Whenever $\supp\mu\subset\{x:\omega\cdot x\leq a\}$, with
$\omega\in\mathbb S^{d-1}$, we have
$$
 \mu\big(\{x:a-h\leq\omega\cdot x\leq a\}\big)
 \lesssim h^{s/2},\qquad 0<h\leq1,
$$
uniformly in $a$ and $\omega$.
\end{lem}

\begin{proof}
Let $\nu=(x\mapsto\omega\cdot x)_*\mu$. Then $\nu$ is supported in
$(-\infty,a]$ and
$|\widehat\nu(t)|\lesssim(1+|t|)^{-s/2}$. We choose a smooth test
function which is at least one near this endpoint. We may modify it
beyond the endpoint without changing its integral against $\nu$, and
use this freedom to make its Fourier transform vanish to high order at
zero.

Choose an integer $L\geq0$ with $L>s/2-1$ and a nonnegative function
$\chi\in C_c^\infty((-2,1))$ equal to one on $[-1,0]$. Set
$$
 \varphi(u)=\sum_{j=0}^L(-1)^j\binom Lj\chi(u-3j).
$$
Every translate with $j\geq1$ is supported in $(0,\infty)$, so
$\varphi=\chi$ on $(-\infty,0]$. Moreover,
$$
 \widehat\varphi(z)=(1-e^{-3iz})^L\widehat\chi(z),
$$
which is $O(|z|^L)$ at zero and rapidly decreasing at infinity.
Positivity and Fourier inversion now give
$$
\begin{aligned}
 \nu([a-h,a])
 &\leq\int_\R\varphi\big((u-a)/h\big)\,\nu(du)\\
 &\lesssim h\int_\R|\widehat\varphi(ht)|(1+|t|)^{-s/2}\,dt\\
 &\leq h^{s/2}\int_\R|\widehat\varphi(z)||z|^{-s/2}\,dz
 \lesssim h^{s/2}.
\end{aligned}
$$
The last integral is finite because $L-s/2>-1$. All constants are
independent of $a$ and $\omega$.
\end{proof}

The exponent $2$ in \cref{thm:finite-perimeter} is sharp whenever
the boundary contains a genuinely planar piece. Recall that $\partial D$
contains a \emph{flat open patch} if there are an affine hyperplane $H$ and
a nonempty relatively open set $\Gamma\subset\partial D$ such that
$\Gamma\subset H$.

\begin{proof}[Proof of \cref{thm:flat-boundary}]
By \cref{prop:smooth-transfer}, we may work with the finite-range
model. The lower bound follows from \cref{thm:finite-perimeter},
and \eqref{eq:dimbound} gives the upper bound $D_{\gamma,d}$. It remains
to prove the upper bound by $2$ when $D_{\gamma,d}>2$. This implies
$\gamma^2<2p$, where $p=d-1$.

Choose orthonormal coordinates $(t,y)\in\R\times\R^p$ near the flat
patch, with the first axis pointing into $D$. There are $a>0$ and a
bounded open set $V\subset\R^p$ such that
$$
 D\cap\big((-a,a)\times V\big)=(0,a)\times V.
$$
Choose a closed nondegenerate rectangle $F\Subset V$, $0<b<a$, and
$0\leq\Phi\in C_c^\infty((-a,a)\times V)$ equal to one on
$[-b,b]\times F$. Set $\mu=\Phi M_{\gamma,D}$ and
$$
 T_h=\frac1hM_\gamma\big((0,h)\times F\big),
 \qquad 0<h<b.
$$
Observe that the definition of the chaos implies $\E[T_h]=|F|$. 
We claim that, for $1<q<2p/\gamma^2$,
$$
 \sup_{0<h<b}\E[T_h^q]<\infty.
$$
To see this, make the change of variables $t=hu$ in the regularized integral.
The logarithmic cutoff estimate and
$|(hu-hv,y-z)|\geq|y-z|$ give
$$
 K_\varepsilon\big((hu,y),(hv,z)\big)
 \leq K_\varepsilon\big((0,y),(0,z)\big)+C
$$
uniformly in $0<h<b$ and $0<\varepsilon<1$. Compare with the field on
$\{0\}\times V$, extended constantly in $u$, plus an independent
Gaussian constant of variance $C$. Kahane's inequality and
\cref{prop:gmcmoments} give the uniform $q$-moment bound
for the regularizations. Fatou's lemma gives the same bound for $T_h$.

The event $\{\dim_F(M_{\gamma,D})>2\}$ has probability zero or one. Indeed,
for every integer $m\geq1$,
$$
 M_{\gamma,D}
 =\rho_{>2^{-m}}(\1_D M_{\leq2^{-m}}).
$$
The measure in brackets is measurable with respect to the
completed $\sigma$-field generated by
$X_{2^{-(j+1)},2^{-j}}$, $j\geq m$. The multiplier and its reciprocal
are smooth near $\overline D$, so the two measures have the same
Fourier dimension. The
event is therefore a tail event, and Kolmogorov's zero--one law applies.

Suppose that this event has probability one. Multiplication by $\Phi$
preserves the asserted Fourier decay. Since $\mu$ is supported in
$\{t\geq0\}$, \cref{lem:one-sided-mass} then gives, almost surely for some $s>2$,
$$
 0\leq T_h
 \leq\frac{\mu(\{(t,y):0\leq t\leq h\})}{h}
 \lesssim h^{s/2-1}\longrightarrow0
$$
as $h\to 0$.
The uniform $L^q$ bound makes this family uniformly integrable, so
$\E[T_h]\to0$, contradicting $\E[T_h]=|F|>0$. This proves the
upper bound by $2$.
\end{proof}

For a flat patch, geometry enters only through this obstruction. For
everywhere positive curvature, it also improves the lower bound.  We
therefore keep the full two-sided argument together in the next section.

\section{Uniformly strictly convex domains} \label{sec:5}

Let $D$ be a bounded convex domain with $C^2$ boundary $S=\partial D$, and
let $n$ be the outward unit normal. For tangent vectors $v,w\in T_xS$,
we use the second fundamental form
$$
 \mathrm{II}_x(v,w)=\langle D_vn,w\rangle,
$$
where $D_vn$ is the derivative of the normal along the surface in the
direction $v$. We call $D$
\emph{uniformly strictly convex} if $\mathrm{II}_x$ is positive definite
for every $x\in S$.  By compactness, its principal curvatures are then
bounded above and below by positive constants.  The geometric facts below
require only $C^2$ regularity, while the lower Fourier bound uses the
$C^\infty$ hypothesis in the theorem.

By \cref{prop:smooth-transfer}, it is enough to prove both
\cref{thm:curved-domain,thm:curved-surface} for the
finite-range field.  We work with that field throughout this section
and transfer the result back to every smooth covariance perturbation at the
end. If $D_{\gamma,d}\leq2$, \cref{thm:finite-perimeter} and the
correlation dimension upper bound already give \cref{thm:curved-domain}.
We therefore
assume $D_{\gamma,d}>2$ when proving \cref{thm:curved-domain}.  The formula for
$D_{\gamma,d}$ then gives $\gamma^2<2(d-1)$, so the surface chaos used
below is subcritical.
The upper bound comes from unusually heavy supporting caps.  For the lower
bound, integration by parts reduces the problem to oscillatory integrals
over $S$, where curvature leaves only two stationary points in each
direction.  The normal integration contributes the exponent $2$, while the
largest possible surface mass near those points contributes
$\alpha_{\gamma,d-1}$.

We collect here the geometric facts used in the proof. For $x\in S$, convexity gives the supporting half-space
$$
 \overline D\subset\{y:n(x)\cdot(y-x)\leq0\}.
$$
The \emph{supporting cap of normal thickness $h$ at $x$} is then
$$
 C_h(x)=\{y\in\overline D:0\leq n(x)\cdot(x-y)\leq h\}.
$$
Thus $C_h(x)$ is the portion of $\overline D$ between its supporting
hyperplane at $x$ and the parallel hyperplane at distance $h$ inside $D$.

\begin{lem}[Uniform strict convexity]
\label{lem:strict-convex-geometry}
Assume that $D$ is bounded and convex, and that its $C^2$ boundary is
uniformly strictly convex.  Then the following statements hold with
constants uniform in the base point and direction.
\begin{enumerate}[label=\textup{(\roman*)}]
 \item The Gauss map $n:S\to\mathbb S^{d-1}$ is a $C^1$
 diffeomorphism.
 \item In orthonormal coordinates centered at $x\in S$, with
 $n(x)=e_d$ and $D$ below its supporting hyperplane, there is a uniform
 neighbourhood $U_x$ on which
 $$
  S\cap U_x=\{(u,-g_x(u)):|u|<r_0\},
  \qquad
  c|u|^2\leq g_x(u)\leq C|u|^2.
 $$
 \item For all sufficiently small $h>0$, the supporting cap of normal
 thickness $h$ at $x$ has diameter $O(h^{1/2})$ and volume comparable to
 $h^{(d+1)/2}$.
 \item For $\omega\in\mathbb S^{d-1}$, the height function
 $\phi_\omega(x)=\omega\cdot x$ on $S$ has exactly two critical points
 $x_+(\omega)$ and $x_-(\omega)$, both uniformly nondegenerate, and
 $$
  |\nabla_S\phi_\omega(x)|
  \asymp
  \dist_S\big(x,\{x_+(\omega),x_-(\omega)\}\big).
 $$
\end{enumerate}
\end{lem}

\begin{proof}
Since by definition
$
 \langle Dn_xv,w\rangle=\mathrm{II}_x(v,w)
$,
the curvature bounds make $Dn_x$ uniformly invertible.
The inverse function theorem therefore makes $n$ a local
diffeomorphism with uniformly controlled inverse. Every direction is attained by maximizing its height function on $\overline{D}$, so $n$ is surjective. If two distinct points had the same outward normal, they would lie in the same supporting hyperplane, and convexity would place their joining segment in $S$, contradicting positive definiteness of the second fundamental form. Thus $n$ is bijective, and its local inverses give a global $C^1$ inverse.

In the coordinates of \textup{(ii)}, we have $g_x(0)=0$ and
$Dg_x(0)=0$. The curvature bounds and continuity of the Hessian give
uniform positive lower and upper bounds for $D^2g_x$ on smaller charts.
Taylor's formula therefore gives
$$
 g_x(u)=\int_0^1(1-v)u^{\mathsf T}D^2g_x(vu)u\,dv,
 \qquad c|u|^2\leq g_x(u)\leq C|u|^2.
$$
Compactness allows the chart radius and constants to be chosen uniformly
in $x$, proving \textup{(ii)}.

For all sufficiently small $h$, the entire cap $C_h(x)$ lies in this
chart, uniformly in $x$. Indeed, the supporting plane meets
$\overline D$ only at $x$, and compactness gives a positive lower bound
for the normal depth outside any fixed neighbourhood of $x$.
The normal thickness of the cap above $u$ is $(h-g_x(u))_+$, so
$$
 |C_h(x)|=\int_{|u|<r_0}(h-g_x(u))_+\,du.
$$
The integrand vanishes unless $|u|\leq\sqrt{h/c}$ and is at most $h$.
On $|u|\leq\sqrt{h/(2C)}$, it is at least $h/2$. Hence
$|C_h(x)|\asymp h^{(d+1)/2}$. Moreover,
$$
 C_h(x)\subset
 \{(u,v):|u|\leq\sqrt{h/c},\ -h\leq v\leq0\},
$$
so its diameter is $O(\sqrt h)$, proving \textup{(iii)}.

Finally, $\nabla_S\phi_\omega=\omega-(\omega\cdot n)n$, so its zeros
are the two points for which $n=\pm\omega$. At these points,
$$
 \operatorname{Hess}_S\phi_\omega
 =-(\omega\cdot n)\mathrm{II}=\mp\mathrm{II},
$$
which proves uniform nondegeneracy. If $\vartheta$ is the angle between
$n(x)$ and $\omega$, then $|\nabla_S\phi_\omega(x)|=|\sin\vartheta|$.
The bi-Lipschitz property of the Gauss map and
$|\sin\vartheta|\asymp\min\{\vartheta,\pi-\vartheta\}$ give the
distance comparison in \textup{(iv)}.
\end{proof}

\subsection{The upper bound}\label{sec:curved-upper}

The correlation dimension formula gives
$\dim_F(M_{\gamma,D})\leq D_{\gamma,d}$.
We prove the remaining bound by $2+\alpha_{\gamma,d-1}$, for which
$C^2$ boundary regularity is enough.

For $x\in\partial D$, let
$C_N(x)=C_{N^{-2}}(x)
$ denote the supporting cap of thickness $N^{-2}$ at $x$.
\cref{lem:strict-convex-geometry}(iii) gives, uniformly in $x$,
\begin{equation}
\label{eq:curved-upper-cap-geometry}
 \operatorname{diam}C_N(x)\lesssim N^{-1},
 \qquad |C_N(x)|\asymp N^{-(d+1)}.
\end{equation}
Moreover, \cref{lem:one-sided-mass}, applied to the
supporting half-space at $x$, shows that Fourier decay with exponent $s$
would imply
$$
 \sup_{x\in\partial D}M_{\gamma,D}(C_N(x))\lesssim N^{-s}.
$$
It therefore suffices to produce heavy caps. We use
the work \cite{DingRoyZeitouni} on the maximum of log-correlated Gaussian
fields. Recall the
multiplicative splitting from \cref{sec:finite-range-model},
\begin{equation}
\label{eq:curved-upper-multiplicative-splitting}
 M_{\gamma,D}=\rho_{>1/N}M_{\leq1/N}\quad\hbox{on }D.
\end{equation}

\begin{lem}
\label{lem:curved-upper-boundary-maximum}
There are finite deterministic sets $\Gamma_N\subset\partial D$ and
$\cF_{>1/N}$-measurable points ${x_N\in\Gamma_N}$ such that, for every
$b<\sqrt{2(d-1)}$,
$$
 \PP\big(\inf_{y\in C_N(x_N)}X_{>1/N}(y)\geq b\log N\big)
 \longrightarrow1.
$$
\end{lem}

\begin{proof}
Choose a bi-Lipschitz boundary chart $\psi$ defined near
$[0,1]^{d-1}$, and set
$$
 V_N=\{0,\ldots,N-1\}^{d-1},\qquad
 \Gamma_N=\{\psi(v/N):v\in V_N\}.
$$
The logarithmic cutoff estimate and the distance comparison for $\psi$ give, uniformly in $v,w\in V_N$,
$$
 \Cov\big(X_{>1/N}(\psi(v/N)),X_{>1/N}(\psi(w/N))\big)
 =\log N-\log_+|v-w|+O(1).
$$
Thus assumptions \emph{(A.0)} and \emph{(A.1)} of
\cite{DingRoyZeitouni} hold, with lattice dimension $d-1$.
By \cite[Theorem~1.2]{DingRoyZeitouni}, the maximum on $\Gamma_N$ is
$\sqrt{2(d-1)}\log N+o_\PP(\log N)$.
Choose $x_N$ to be its first maximizer in a fixed ordering, so it is
$\cF_{>1/N}$-measurable.

By \eqref{eq:dercovest}, each $\partial_iX_{>1/N}$ has variance
$O(N^2)$ and canonical metric bounded by $CN^2|x-y|$.
\cref{lem:dudley} therefore gives
$$
 \E\big[\|\nabla X_{>1/N}\|_{L^\infty(\overline D)}\big]
 \lesssim N\sqrt{\log N}.
$$
Together with \eqref{eq:curved-upper-cap-geometry} and Markov's
inequality, this yields
$$
 \sup_{x\in\partial D}\sup_{y\in C_N(x)}
 |X_{>1/N}(y)-X_{>1/N}(x)|
 =O_\PP(\sqrt{\log N})=o_\PP(\log N).
$$
Combining this with the maximum estimate proves the claim.
\end{proof}

\begin{lem}
\label{lem:curved-upper-fine-cap}
There are $c_0,c_1>0$ such that, uniformly in $x\in\partial D$ and
all sufficiently large $N$,
$$
 \PP\big(M_{\leq1/N}(C_N(x))\geq c_0N^{-(d+1)}\big)\geq c_1.
$$
\end{lem}

\begin{proof}
Fix $1<q<2(d-1)/\gamma^2$. We show that
$$
 \E[M_{\leq1/N}(C_N(x))]\asymp N^{-(d+1)}\quad \text{and} \quad
 \E[M_{\leq1/N}(C_N(x))^q]\lesssim N^{-q(d+1)}.
$$
The first estimate follows from normalization and
\eqref{eq:curved-upper-cap-geometry}. For the second, regularize the
fine field at scale $\varepsilon<N^{-1}$ and, in the coordinates of \cref{lem:strict-convex-geometry}\textup{(ii)}, set
$$
 (u,t)=(N^{-1}v,-N^{-2}w).
$$
The rescaled cap lies in a fixed cylinder $B(0,R)\times[0,1]$, with
Jacobian $N^{-(d+1)}$. Since $|y-z|\geq N^{-1}|v-v'|$, its regularized
fine covariance satisfies
$$
\begin{aligned}
 K_{\varepsilon,\leq1/N}(y,z)
 \leq\log_+\frac1{|v-v'|\vee N\varepsilon}+C
 \leq K_{N\varepsilon}^{\mathrm{fr},d-1}(v,v')+C',
\end{aligned}
$$
with constants independent of $\varepsilon,x,N$.
The comparison field is the $(d-1)$-dimensional finite-range field as in \cref{sec:finite-range-model},
constant in $w$, plus an independent Gaussian constant.
Kahane's inequality, integration in $w$ and
\cref{prop:gmcmoments} give the asserted $q$-moment bound
uniformly in $\varepsilon$, while Fatou's lemma removes the regularization.

For $Z=M_{\leq1/N}(C_N(x))$ and $m=\E[Z]$, H\"older's inequality gives
$$
 m/2\leq\E[Z\1_{\{Z\geq m/2\}}]
 \leq\E[Z^q]^{1/q}\PP(Z\geq m/2)^{1-1/q}.
$$
The two moment estimates therefore imply the desired uniform positive
probability.
\end{proof}

Fix $a>2+\alpha_{\gamma,d-1}$ and choose $b<\sqrt{2(d-1)}$ such that
$$
 d+1+\frac{\gamma^2}{2}-\gamma b<a.
$$
On an event whose probability tends to one,
\cref{lem:curved-upper-boundary-maximum} gives
$\inf_{C_N(x_N)}\rho_{>1/N}\geq N^{\gamma b-\gamma^2/2}$.
Conditionally on $\cF_{>1/N}$, the centre $x_N$ is fixed and the fine
field is independent. \cref{lem:curved-upper-fine-cap} and
\eqref{eq:curved-upper-multiplicative-splitting} consequently give
$$
 \liminf_{N\to\infty}
 \PP\big(Z_N\geq N^{-a}\big)>0,
 \qquad
 Z_N=\max_{x\in\Gamma_N}M_{\gamma,D}(C_N(x)).
$$
The maxima are over finite deterministic sets and are measurable.
Reverse Fatou gives positive probability that $Z_N\geq N^{-a}$
infinitely often. On that event, the cap bound above rules out every
Fourier decay exponent $s>a$, so
$$
 \PP\big(\dim_F(M_{\gamma,D})\leq a\big)>0.
$$

Finally, deleting any finite number of independent scale blocks changes
$M_{\gamma,D}$ by a smooth strictly positive multiplier on
$\overline D$. The multiplier argument in
\cref{prop:smooth-transfer} shows that its Fourier dimension
is unchanged. Hence $\dim_F(M_{\gamma,D})$ is measurable with respect to
the completed tail $\sigma$-field of those blocks, and Kolmogorov's
zero--one law upgrades the last probability to one.
Letting $a\to2+\alpha_{\gamma,d-1}$ along a countable sequence and
combining with the correlation bound proves
$
 \dim_F(M_{\gamma,D})
 \leq\min\{D_{\gamma,d},2+\alpha_{\gamma,d-1}\}
$ almost surely.

\subsection{The lower bound}\label{sec:curved-lower}

We prove the lower bound
$$
 \dim_F(M_{\gamma,D})
 \geq \min\{D_{\gamma,d},2+\alpha_{\gamma,d-1}\}.
$$
Fix
$
 s<\min\{D_{\gamma,d},\,2+\alpha_{\gamma,d-1}\}.
$
The case $s\leq2$ follows already from
\cref{thm:finite-perimeter}, so we assume $s>2$ and put
$
 a=s-2<\alpha_{\gamma,d-1}.
$
We use the joint realization and augmented filtration of \cref{sec:finite-range-model}, and write
$$
 Y_t=X_{>e^{-t}},
 \qquad
 \rho_t=\rho_{>e^{-t}},
 \qquad
 \cF_t=\cF_{>e^{-t}}.
$$

We first reduce the lower bound to a surface estimate.  The only difference
from the torus argument is the presence of the boundary terms created by
integration by parts.
For $\xi\neq0$, let $\omega=\xi/|\xi|$ and
$\partial_\omega=\omega\cdot\nabla$.  Repeated use of the divergence theorem
gives, for every integer $J\geq1$,
\begin{align}
 \int_D e^{-i\xi\cdot x}\rho_{>\delta}(x)\,dx
 &=\frac{1}{(i|\xi|)^J}
   \int_D e^{-i\xi\cdot x}
   \partial_\omega^J\rho_{>\delta}(x)\,dx
 \notag\\
 &\quad-
 \sum_{\ell=0}^{J-1}\frac{1}{(i|\xi|)^{\ell+1}}
 \int_S e^{-i\xi\cdot x}(\omega\cdot n(x))
 \partial_\omega^\ell\rho_{>\delta}(x)\,d\sigma(x).
 \label{eq:domain-ibp-lower}
\end{align}
As before, set
$$
 N^D_{\delta,J}
 =\sum_{|\beta|\leq J}\delta^{|\beta|}
 \int_D|\partial^\beta\rho_{>\delta}(x)|\,dx.
$$
The volume term in \eqref{eq:domain-ibp-lower} is bounded by
$(|\xi|\delta)^{-J}N^D_{\delta,J}$.  As in the finite perimeter proof,
\cref{lem:localmoments} at $u=1$ gives summable expectations for its
annular norms once the scales $\delta_n$ below are chosen and $J$ is
sufficiently large.
For $0\leq\ell\leq J-1$, define the normalized boundary transform
$$
 I_{\ell,\delta}(\xi)
 =\delta^\ell\int_S e^{-i\xi\cdot x}
 (\omega\cdot n(x))\partial_\omega^\ell\rho_{>\delta}(x)
 \,d\sigma(x).
$$
For $\eta>0$, let
$
 \delta_n=2^{-\lfloor(1-\eta)n\rfloor}
$ as before.
Since
$|\xi|\delta_n\gtrsim2^{\eta n}$ on $A_n^\R$, the $\ell$-th boundary term in
\eqref{eq:domain-ibp-lower} satisfies
$$
 \Big\|
 |\xi|^{-1}(|\xi|\delta_n)^{-\ell}
 I_{\ell,\delta_n}(\xi)
 \Big\|_{B_{k,s,n}^\R}
 \lesssim
 \|I_{\ell,\delta_n}\|_{B_{k,s-2,n}^\R}.
$$
Thus it remains to prove
\begin{equation}\label{eq:remaining-boundary-series}
 \sum_{n\geq0}
 \|I_{\ell,\delta_n}\|_{B_{k,s-2,n}^\R}^{2k}<\infty,
 \qquad 0\leq\ell\leq J-1,
\end{equation}
almost surely.
The fine fluctuation $\widetilde M_{\leq\delta_n,D}$ is summable in
$B_{k,s,n}^\R$ by \cref{lem:domain-fine-part}.  Thus
\eqref{eq:remaining-boundary-series} is the only new estimate.

\subsubsection{Good events for local surface masses}
The following surface estimates hold throughout $\gamma^2<2(d-1)$,
independently of the assumption $s>2$ in the argument above.
They will also be used in the proof of \cref{thm:curved-surface}.

The derivatives $\partial^\beta$ below are derivatives in the
ambient space $\R^d$, evaluated on $S$.  For $x\in S$ and $u>0$, set
$
 B_S(x,u)=S\cap B(x,u),
$
where $B(x,u)$ is the Euclidean ball in $\R^d$ of center $x$ and
radius $u$.
For an integer $J\geq0$, define
$$
 N^S_{r,J}(x,u)
 =\sum_{|\beta|\leq J}r^{|\beta|}
 \int_{B_S(x,u)}|\partial^\beta\rho_{>r}(y)|\,d\sigma(y).
$$

Fix $0<b<\alpha_{\gamma,d-1}$ and $\varepsilon>0$.  For $n\geq1$,
$0<\delta\leq1$, and $J\geq0$, let
$$
 \mathcal G_n(\delta,J)
 =
 \Big\{
 N^S_{r,J}(x,u)
 \leq C_0 2^{n\varepsilon}u^b
 \text{ for all }\delta\leq r\leq u\leq1,\ x\in S
 \Big\},
$$
where $C_0$ is a sufficiently large deterministic constant.

\begin{lem}[Uniform local surface bounds]
\label{lem:marked-frostman-event}
Set $q=\sqrt{2(d-1)}/\gamma$.  For every $J\geq0$,
$$
 \PP\big(\mathcal G_n(\delta,J)^c\big)
 \lesssim_{\varepsilon,b,q,J,C_0}
 (1+|\log\delta|)2^{-n\varepsilon q}.
$$
Consequently, if $\delta_n\geq2^{-n}$, then
$\mathcal G_n(\delta_n,J)$ occurs for all sufficiently large $n$, almost
surely.
\end{lem}

\begin{proof}
Since $\gamma^2<2(d-1)$,
$$
 1<q<\frac{2(d-1)}{\gamma^2},
 \qquad
 \zeta_{d-1}(q)-(d-1)-bq
 =q(\alpha_{\gamma,d-1}-b)>0.
$$
We first show that, uniformly for $r\leq u\leq1$ and $x\in S$,
$$
 \E\big[N^S_{r,J}(x,u)^q\big]
 \lesssim_{q,J}u^{\zeta_{d-1}(q)}.
$$
Let
$\mathbf D_r=(r^{|\beta|}\partial^\beta X_{>r})_{1\leq|\beta|\leq J}$.
Fa\`a di Bruno's formula gives
$$
 \sum_{|\beta|\leq J}r^{|\beta|}
 |\partial^\beta\rho_{>r}(y)|
 \lesssim_{\gamma,J}
 \rho_{>r}(y)(1+|\mathbf D_r(y)|)^J.
$$
For $J=0$, the right-hand side is simply $\rho_{>r}(y)$.
The polynomial factor is bounded by a finite sum of
$e^{v\cdot\mathbf D_r}$, with fixed vectors $v$.
By \eqref{eq:dercovest}, the kernels
$\Cov(\mathbf D_r(y),\mathbf D_r(z))$ and
$\Cov(\mathbf D_r(y),X_{>r}(z))$ are uniformly bounded.
Thus each Gaussian field $\gamma X_{>r}+v\cdot\mathbf D_r$ has
covariance differing from that of $\gamma X_{>r}$ by a bounded kernel,
and its exponential normalization differs by a bounded factor.
Kahane's inequality and \cref{prop:gmcmoments} prove the
displayed moment estimate.

The identity $\E[\rho_t\mid\cF_v]=\rho_v$ can be differentiated
spatially at positive cutoffs, so $\partial^\beta\rho_t(y)$ is a
martingale. Conditional Jensen and Fubini show that
$\int_{B_S(x,u)}|\partial^\beta\rho_t(y)|\,d\sigma(y)$ is a
nonnegative submartingale. Since powers of $r$ are comparable on one
dyadic interval, Doob's inequality (\cref{lem:martingale-inequalities}\textup{(i)}) and the preceding estimate at its
terminal cutoff give, for $r_0/2\leq u\leq1$,
$$
 \E\Big[\sup_{r_0/2\leq r\leq r_0}
 N^S_{r,J}(x,u)^q\Big]
 \lesssim_{q,J}u^{\zeta_{d-1}(q)}.
$$
Test dyadic radii $u$ and $u$-nets of $S$, each with $O(u^{-(d-1)})$
centers. Fixed enlargements of the balls cover every center and
intermediate radius. Markov's inequality bounds the failure probability
for each radius and scale interval by
$$
 C2^{-n\varepsilon q}
 u^{\zeta_{d-1}(q)-(d-1)-bq}.
$$
The positive exponent makes the sum over radii finite. There are
$O(1+|\log\delta|)$ dyadic cutoff intervals between $\delta$ and one,
which proves the claim. If $\delta_n\geq2^{-n}$, the probabilities are
summable in $n$, and Borel--Cantelli finishes the proof.
\end{proof}

\subsubsection{Curved surface estimate}
For $\omega\in\mathbb S^{d-1}$, let
$\phi_\omega(x)=\omega\cdot x$ on $S$.  We use the notation
$x_+(\omega),x_-(\omega)$ for its two critical points.  By
\cref{lem:strict-convex-geometry}(iv),
$$
 |\nabla_S\phi_\omega(x)|
 \asymp
 \dist_S(x,\{x_+(\omega),x_-(\omega)\})
$$
uniformly in $x$ and $\omega$.  This is the geometric input in the proof of
\eqref{eq:remaining-boundary-series}.

\begin{prop}\label{prop:curved-surface-estimate}
Assume that $\gamma^2<2(d-1)$. Let
$0<a<\alpha_{\gamma,d-1}$ and let $J,k\geq1$ be integers.
There exists $\eta_0>0$, independent of $J$ and $k$, such that the
following holds. For every $0<\eta<\eta_0$ and
$
\delta_n=2^{-\lfloor(1-\eta)n\rfloor},
$
almost surely, simultaneously for every $0\leq\ell\leq J-1$,
$$
 \sum_{n\geq1}
 \|I_{\ell,\delta_n}\|_{B_{k,a,n}^\R}^{2k}<\infty.
$$
The conclusion also holds when the factor $\omega\cdot n(x)$ in
$I_{\ell,\delta}$ is replaced by any fixed
$A\in C^\infty(S\times\mathbb S^{d-1})$.
\end{prop}

\begin{proof}
We prove the statement with $A_\omega(x)=A(x,\omega)$. The choices used
below are $A_\omega(x)=\omega\cdot n(x)$ for the boundary terms and
$A_\omega\equiv1$ for surface chaos. Constants may depend on the fixed surface and parameters, on $J,k$,
and on finitely many derivatives of $A$, but are uniform in $n$,
$\xi\in A_n^\R$, the intermediate scale, and
$0\leq\ell\leq J-1$. The good event $\mathcal G_n$ controls all
surface centers simultaneously and does not depend on $\xi$.

We reveal the field while expanding a smooth cutoff around the two
stationary points. At scale $r$, its radius is comparable to $h/r$, until
it covers the whole surface. The relation $r(h/r)=h$ will give the desired
gain in the quadratic variation. The derivative of the cutoff is
supported away from the stationary points, where we use tangential
integration by parts.

Choose $a<b<\alpha_{\gamma,d-1}$ and $\vartheta,\varepsilon>0$ so small
that
$$
 b(1-\vartheta)>a+2\varepsilon.
$$
Take $\eta_0=\vartheta/2$, choose an integer $K$ with
$K\vartheta>a/2+\varepsilon$, and put $J_0=J+K$.
These choices are independent of $k$, and $\eta_0$ is independent of $J$.
Fix $0<\eta<\eta_0$, $0\leq\ell\leq J-1$ and $\xi\in A_n^\R$, and write
$$
 R=2^n,\qquad h=R^{-(1-\vartheta)},\qquad
 \delta=\delta_n,\qquad
 t_0=\log\frac1{\sqrt h},\qquad t_1=\log\frac1\delta.
$$
For all sufficiently large $n$, $\delta\leq h$.
We work on the event
$\mathcal G_n=\mathcal G_n(\delta,J_0)$ from
\cref{lem:marked-frostman-event}, and abbreviate
$L=C_0 2^{n\varepsilon}$.
Fixed enlargements of surface balls are harmless: when their radius
exceeds one, use a fixed finite cover of $S$ by unit balls.

Let $\chi\in C^\infty([0,\infty))$ be nonincreasing, equal to one on
$[0,1]$ and zero on $[4,\infty)$. For $t_0\leq t\leq t_1$, put
$r=e^{-t}$, $\omega=\xi/|\xi|$ and
$$
 c_t(x)
 =\chi\big(\frac{r^2}{h^2}|\nabla_S\phi_\omega(x)|^2\big).
$$
The squared gradient makes this a smooth function even at the stationary
points. With $\lambda=h/r$, the cutoff equals one where
$|\nabla_S\phi_\omega|\leq\lambda$ and zero where
$|\nabla_S\phi_\omega|\geq2\lambda$. Its support expands as $r$ decreases.
By \cref{lem:strict-convex-geometry}\textup{(iv)}, with
$u(r)=\min\{1,h/r\}$,
$$
 \operatorname{supp}c_t
 \subset B_S(x_+(\omega),Cu(r))
 \cup B_S(x_-(\omega),Cu(r)),
$$
where a fixed enlargement covers $S$ when $u(r)$ is comparable to one.
At $t=t_0$, these balls have radius $O(\sqrt h)$.
Starting at $r=\sqrt h$ ensures $r\leq u(r)$ throughout the interval.
Since $|\nabla_S\phi_\omega|\leq1$ and $\delta\leq h$, we have
$c_{t_1}\equiv1$.

Define
$$
 F_t
 =\delta^\ell\int_S e^{-i\xi\cdot x}A_\omega(x)c_t(x)
 \partial_\omega^\ell\rho_t(x)\,d\sigma(x).
$$
Spatial differentiation of $d\rho_t=\gamma\rho_t\,dY_t$ gives
$$
 d(\partial_\omega^\ell\rho_t)
 =\gamma\sum_{j=0}^\ell\binom{\ell}{j}
 \partial_\omega^{\ell-j}\rho_t\,
 d(\partial_\omega^jY_t).
$$
The cutoff $c_t$ is deterministic and has finite variation in time.
Thus It\^o's product rule gives the semimartingale decomposition
$$
 F_t=F_{t_0}+\int_{t_0}^t B_v\,dv+Z_t,
$$
where
$$
 B_t=\delta^\ell\int_S e^{-i\xi\cdot x}A_\omega(x)
 \dot c_t(x)\partial_\omega^\ell\rho_t(x)\,d\sigma(x)
$$
and $Z_{t}$ is the continuous complex martingale defined by $Z_{t_0}=0$ and
$$
 dZ_t
 =\gamma\delta^\ell\sum_{j=0}^\ell\binom{\ell}{j}
 \int_S e^{-i\xi\cdot x}A_\omega(x)c_t(x)
 \partial_\omega^{\ell-j}\rho_t(x)\,
 d(\partial_\omega^jY_t(x))\,d\sigma(x).
$$
Since $r\geq\delta>0$, \eqref{eq:dercovest} and
Gaussian exponential moments bound the required spatial derivatives of
$\rho_t$ in $L^2$, uniformly on this finite time interval and on $S$.
The derivatives of the increment covariance $k_r$ are bounded there as
well. These bounds justify spatial differentiation and stochastic
Fubini. In particular, the real and imaginary parts of $Z$ are continuous
square-integrable martingales.

Since $c_{t_1}=1$, the decomposition gives
$$
 I_{\ell,\delta}(\xi)
 =F_{t_0}+\int_{t_0}^{t_1}B_t\,dt+Z_{t_1}.
$$
On $\mathcal G_n$, the initial support estimate and
$\delta\leq\sqrt h$ imply
$$
 |F_{t_0}|\lesssim Lh^{b/2}.
$$

The term $B_t$ comes entirely from the changing cutoff. When $r\leq h$,
we have $c_t\equiv1$, so $B_t=0$.
For $h\leq r\leq\sqrt h$, put $\lambda=h/r$.
If $z=|\nabla_S\phi_\omega|^2/\lambda^2$, then
$\dot c_t=-2z\chi'(z)$. Consequently,
$$
 \operatorname{supp}\dot c_t
 \subset\{\lambda\leq|\nabla_S\phi_\omega|\leq2\lambda\},
 \qquad
 |\nabla_S^m\dot c_t|\lesssim_m\lambda^{-m}.
$$
The derivative bounds follow from nondegeneracy of the stationary
points, uniformly in $\omega$ (for $\lambda$ bounded below they follow
from compactness).
On this support, use the tangential vector field
$$
 V_\omega
 =\frac{\nabla_S\phi_\omega}{|\nabla_S\phi_\omega|^2},
 \qquad
 V_\omega\big(e^{-i|\xi|\phi_\omega}\big)
 =-i|\xi|e^{-i|\xi|\phi_\omega}.
$$
Its formal adjoint is
$V_\omega^*f=-\operatorname{div}_S(V_\omega f)$, and
$|\nabla_S^mV_\omega|\lesssim_m\lambda^{-m-1}$.
Since the amplitude is supported away from the stationary points,
integrating by parts tangentially $K$ times on the closed surface gives
$$
 B_t
 =\frac{\delta^\ell}{(-i|\xi|)^K}
 \int_S e^{-i|\xi|\phi_\omega}
 (V_\omega^*)^K
 \big(A_\omega\dot c_t\,\partial_\omega^\ell\rho_t\big)\,d\sigma.
$$
A term with $j$ additional derivatives on the density has coefficient
bounded by
$|\xi|^{-K}\delta^\ell\lambda^{-(2K-j)}$, where $0\leq j\leq K$.
Tangential derivatives of order $j$ of $\partial_\omega^\ell\rho_t$
are combinations of ambient derivatives of orders at most $\ell+j$,
with bounded geometric coefficients. Since $\delta\leq r\leq\lambda$,
$$
 \delta^\ell\lambda^{-(2K-j)}
 \leq(\lambda r)^{-K}r^{\ell+j}.
$$
Therefore, on $\mathcal G_n$,
$$
 |B_t|
 \lesssim (|\xi|\lambda r)^{-K}
 \big(N^S_{r,J_0}(x_+(\omega),C\lambda)
      +N^S_{r,J_0}(x_-(\omega),C\lambda)\big)
 \lesssim L(Rh)^{-K}\lambda^b.
$$
Here $\lambda r=h$, and $dt=d\lambda/\lambda$ on this interval, so
$$
 \Big|\int_{t_0}^{t_1}B_t\,dt\Big|
 \lesssim L(Rh)^{-K}
 \int_{\sqrt h}^1\lambda^b\frac{d\lambda}{\lambda}
 \lesssim L(Rh)^{-K}.
$$

It remains to estimate $Z$. As in the preliminaries, write
$\langle Z\rangle=\langle\Re Z\rangle+\langle\Im Z\rangle$. Put
$$
 W_r(x)=\sum_{|\beta|\leq J_0}
 r^{|\beta|}|\partial^\beta\rho_{>r}(x)|.
$$
The increment covariance gives, for ambient multi-indices $\mu,\nu$,
$$
 d\langle\partial^\mu Y(x),\partial^\nu Y(y)\rangle_t
 =\partial_x^\mu\partial_y^\nu k_r(x-y)\,dt.
$$
Since $k$ is smooth and supported in $B(0,1)$,
$$
 |\partial_x^\mu\partial_y^\nu k_r(x-y)|
 \lesssim r^{-|\mu|-|\nu|}
 \1_{\{|x-y|\leq r\}}.
$$
For each $j\leq\ell$, the coefficient in $dZ_t$ satisfies
$$
 \delta^\ell r^{-j}|\partial_\omega^{\ell-j}\rho_t|
 =\big(\delta/r\big)^\ell
 r^{\ell-j}|\partial_\omega^{\ell-j}\rho_t|
 \lesssim W_r.
$$
Using $0\leq c_t\leq1$ and the boundedness of $A$, we obtain
$$
 d\langle Z\rangle_t
 \lesssim\int_{\operatorname{supp}c_t}W_r(x)
       \int_{B_S(x,r)}W_r(y)\,d\sigma(y)\,d\sigma(x)\,dt
 \lesssim L^2r^bu(r)^b\,dt
$$
on $\mathcal G_n$.
Indeed, the inner integral is at most $Lr^b$, and the integral over the
support of $c_t$ is at most a constant times $Lu(r)^b$.
Both estimates follow from the good event because $r\leq u(r)$.
Consequently, uniformly in $\xi\in A_n^\R$, on $\mathcal{G}_n$,
$$
 \langle Z\rangle_{t_1}
 \lesssim L^2\int_\delta^{\sqrt h}
 r^b\min\{1,h/r\}^b\frac{dr}{r}
 \lesssim L^2h^b\big(1+|\log h|\big)
 \lesssim nL^2h^b.
$$
Thus $\mathcal G_n$ is contained in the terminal event
$\{\langle Z\rangle_{t_1}\leq CnL^2h^b\}$.
\cref{lem:martingale-inequalities}\textup{(iii)} applies to this
event.
Combining its moment bound with the estimates for $F_{t_0}$ and
$\int B_t\,dt$ gives
$$
 \E\big[\1_{\mathcal G_n}|I_{\ell,\delta}(\xi)|^{2k}\big]
 \lesssim_k n^kL^{2k}h^{bk}+L^{2k}(Rh)^{-2kK}.
$$
The construction is jointly measurable in the random outcome and $\xi$. Tonelli's theorem and
$\int_{A_n^\R}|\xi|^{ka-d}\,d\xi\asymp R^{ka}$ now give
$$
 \E\big[\1_{\mathcal G_n}
 \|I_{\ell,\delta_n}\|_{B_{k,a,n}^\R}^{2k}\big]
 \lesssim_k n^k2^{-nk(b(1-\vartheta)-a-2\varepsilon)}
 +2^{-nk(2K\vartheta-a-2\varepsilon)}.
$$
Both exponents are positive, so these expectations are summable.
\cref{lem:marked-frostman-event} shows that $\mathcal G_n$ occurs
eventually almost surely, allowing the indicators to be removed from
the tail. The finitely many initial annular norms are finite because
their regularization scales $\delta_n$ are positive. Taking the finite intersection over
$0\leq\ell\leq J-1$ completes the proof.
\end{proof}

\subsubsection{Completing the proof}
Choose $\eta>0$ satisfying both
\cref{lem:domain-fine-part} at exponent $s$ and
\cref{prop:curved-surface-estimate} at exponent $s-2$,
put $\delta_n=2^{-\lfloor(1-\eta)n\rfloor}$, and choose $J$ with
$J\eta>s/2$.
For each fixed $k$, the volume and fine terms have summable first annular
moments, hence almost surely summable $2k$-th powers. The surface
proposition gives the same $2k$-summability for the finitely many boundary
terms in \eqref{eq:domain-ibp-lower}. Therefore
$$
 \sum_n\|\widehat M_{\gamma,D}\|_{B_{k,s,n}^\R}^{2k}<\infty
 \qquad\text{almost surely}.
$$
Countable intersections over $k$ and $s$, followed by
\cref{lem:weighted-fourier-criterion} and smooth transfer, complete
the proof of \cref{thm:curved-domain}.

\subsection{Proof of \texorpdfstring{\cref{thm:curved-surface}}{Theorem~\getrefnumber{thm:curved-surface}}}

Put $p=d-1$, the dimension of $S$. By
\cref{prop:smooth-transfer}, it is enough to work with the
finite-range field and reference measure $\sigma_S$. Write
$$
 M^S_{>\delta}(dx)=\rho_{>\delta}(x)\,d\sigma(x),
 \qquad
 \widetilde M^S_{\leq\delta}=M_\gamma^S-M^S_{>\delta}.
$$
Fix $0<s<\alpha_{\gamma,p}$. Since
$\alpha_{\gamma,p}<D_{\gamma,p}$, choose $1<q\leq2$ such that
$$
 q<2p/\gamma^2,
 \qquad
 \zeta_p(q)-p>sq/2.
$$
Take $\eta>0$ small enough that
$(1-\eta)(\zeta_p(q)-p)>sq/2$ and that
\cref{prop:curved-surface-estimate} applies with $a=s$.
Put $\delta_n=2^{-\lfloor(1-\eta)n\rfloor}$.

The fine fluctuation estimate of
\cref{lem:domain-fine-part} applies with surface measure. Indeed,
only $O(\delta^{-p})$ cubes of side $\delta$ meet $S$, so the same coloring argument
gives conditional independence, and each piece has $q$-th mass moment
$O(\delta^{\zeta_p(q)})$ by \cref{prop:gmcmoments}. The Fourier norm still uses the ambient dimension $d$. Thus, for
every fixed $k\geq1$,
$$
 \E\Big[
 \|\widehat{\widetilde M^S_{\leq\delta_n}}\|_{B_{k,s,n}^\R}^{q}
 \Big]
 \lesssim_{k,q,S}
 2^{nsq/2}\delta_n^{\zeta_p(q)-p},
$$
which decays geometrically in $n$. Jensen's inequality gives summable
first moments, so the annular norms and their $2k$-th powers are almost
surely summable.
\cref{prop:curved-surface-estimate}, with $A_\omega\equiv 1$ and
$\ell=0$, gives summable $2k$-th powers of the annular norms of
$\widehat M^S_{>\delta_n}$.
\cref{lem:weighted-fourier-criterion}, followed by countable
intersections over $k$ and over $s$ increasing to $\alpha_{\gamma,p}$,
proves the lower bound.

For the upper bound, suppose that
$|\widehat M_\gamma^S(\xi)|\lesssim(1+|\xi|)^{-s/2}$ for some $s>0$.
The quadratic boundary estimate in
\cref{lem:strict-convex-geometry}(ii) gives
$
 0\leq n(x)\cdot(x-y)\lesssim|x-y|^2$ for  $x,y\in S$ sufficiently close.
Consequently, $B_S(x,r)\subset C_{Cr^2}(x)$ for all sufficiently
small $r$, uniformly in $x$.
\cref{lem:one-sided-mass} therefore yields
\begin{equation}\label{eq:surface-ball-from-decay}
 \sup_{x\in S}M_\gamma^S(B_S(x,r))\lesssim r^s.
\end{equation}
Let $A_N(x)=B_S(x,c/N)$, and
choose $c>0$ so small that
$A_N(x)\subset C_N(x)$ uniformly in $x\in S$, where
$C_N(x)$ is the cap of the domain bounded by $S$.
Write $M^S_{\leq1/N}$ for the surface chaos of the fine field, so that
$M_\gamma^S=\rho_{>1/N}M^S_{\leq1/N}$.
Fix $1<q<2p/\gamma^2$. On $A_N(x)\times A_N(x)$,
$K_{>1/N}\geq\log N-C$ for all sufficiently large $N$.
Apply Kahane comparison to regularizations, using the fine field plus an independent Gaussian
constant of variance $\log N-C$ as the comparison field. \cref{prop:gmcmoments} and Fatou's lemma then give
$$
 \E[M^S_{\leq1/N}(A_N(x))^q]
 \lesssim
 N^{-\frac{\gamma^2}{2}(q^2-q)}N^{-\zeta_p(q)}
 =N^{-pq}.
$$
Also $\E[M^S_{\leq1/N}(A_N(x))]=\sigma_S(A_N(x))\asymp N^{-p}$.
The H\"older argument in the proof of
\cref{lem:curved-upper-fine-cap} therefore gives constants
$c_0,c_1>0$, independent of $x,N$, such that
$$
 \PP\big(M^S_{\leq1/N}(A_N(x))\geq c_0N^{-p}\big)\geq c_1.
$$

Fix $a>\alpha_{\gamma,p}$ and choose $b<\sqrt{2p}$ such that
$p+\gamma^2/2-\gamma b<a$.
\cref{lem:curved-upper-boundary-maximum} gives
$\cF_{>1/N}$-measurable points $x_N\in\Gamma_N$ for which
$\inf_{A_N(x_N)}\rho_{>1/N}\geq N^{\gamma b-\gamma^2/2}$
on an event whose probability tends to one.
Conditioning on $\cF_{>1/N}$ and using independence of the fine field,
we obtain
$$
 \liminf_{N\to\infty}
 \PP\Big(\max_{x\in\Gamma_N}M_\gamma^S(A_N(x))\geq N^{-a}\Big)>0.
$$
Reverse Fatou and \eqref{eq:surface-ball-from-decay} imply
$\PP(\dim_F(M_\gamma^S)\leq a)>0$.
Deleting finitely many independent scale blocks multiplies the surface
chaos by a smooth positive ambient function. The same multiplier and
tail zero--one argument as in \cref{sec:curved-upper} therefore gives
$\dim_F(M_\gamma^S)\leq a$ almost surely.
Letting $a\to\alpha_{\gamma,p}$ along a countable sequence proves
the upper bound.

\begin{rmk}
For canonical circular chaos, the covariance differs smoothly from that
of the planar auxiliary field restricted to $S^1$. Apply Gaussian
completion on $\T$ as in \cref{prop:smooth-transfer}.
The resulting smooth positive multipliers extend to an annular
neighbourhood by keeping them constant along radial lines. The
Euclidean convolution argument in that proposition therefore preserves
Fourier dimension after embedding the circle in $\R^2$.
This proves the consequence for the circle stated after
\cref{thm:curved-surface}.
\end{rmk}

\bibliographystyle{amsplain}
\bibliography{biblio}

@book{AdlerTaylor,
  author = {Adler, Robert J. and Taylor, Jonathan E.},
  title = {Random fields and geometry},
  series = {Springer Monographs in Mathematics},
  publisher = {Springer},
  address = {New York},
  year = {2007},
  url = {https://doi.org/10.1007/978-0-387-48116-6}
}

@misc{ArguinHamdan,
  author = {Arguin, Louis-Pierre and Hamdan, Jad},
  title = {On the {Fourier} coefficients of critical {Gaussian} multiplicative chaos},
  howpublished = {arXiv preprint arXiv:2510.24424},
  year = {2025},
  url = {https://arxiv.org/abs/2510.24424}
}

@misc{AtherfoldNajnudel,
  author = {Atherfold, Christopher and Najnudel, Joseph},
  title = {The {Fourier} coefficients of the critical holomorphic multiplicative chaos},
  howpublished = {arXiv preprint arXiv:2508.13849},
  year = {2025},
  url = {https://arxiv.org/abs/2508.13849}
}

@article{BarralKupiainenNikulaSaksmanWebb,
  author = {Barral, Julien and Kupiainen, Antti and Nikula, Miika and Saksman, Eero and Webb, Christian},
  title = {Basic properties of critical lognormal multiplicative chaos},
  journal = {Ann. Probab.},
  volume = {43},
  number = {5},
  year = {2015},
  pages = {2205--2249},
  url = {https://doi.org/10.1214/14-AOP931}
}

@book{BerestyckiPowell,
  author = {Berestycki, Nathana{\"e}l and Powell, Ellen},
  title = {{Gaussian} free field and {Liouville} quantum gravity},
  series = {Cambridge Studies in Advanced Mathematics},
  volume = {220},
  publisher = {Cambridge University Press},
  address = {Cambridge},
  year = {2025},
  url = {https://doi.org/10.1017/9781009405492}
}

@article{BerestyckiWebbWong2018,
  author = {Berestycki, Nathana{\"e}l and Webb, Christian and Wong, Mo Dick},
  title = {Random {Hermitian} matrices and {Gaussian} multiplicative chaos},
  journal = {Probab. Theory Related Fields},
  volume = {172},
  number = {1--2},
  year = {2018},
  pages = {103--189},
  url = {https://doi.org/10.1007/s00440-017-0806-9}
}

@article{Bertacco2023,
  author = {Bertacco, Federico},
  title = {Multifractal analysis of {Gaussian} multiplicative chaos and applications},
  journal = {Electron. J. Probab.},
  volume = {28},
  year = {2023},
  pages = {Paper No. 2, 36},
  url = {https://doi.org/10.1214/22-EJP893}
}

@article{BGKRV,
  author = {Baverez, Guillaume and Guillarmou, Colin and Kupiainen, Antti and Rhodes, R{\'e}mi and Vargas, Vincent},
  title = {The {Virasoro} structure and the scattering matrix for {Liouville} conformal field theory},
  journal = {Probab. Math. Phys.},
  volume = {5},
  number = {2},
  year = {2024},
  pages = {269--320},
  url = {https://doi.org/10.2140/pmp.2024.5.269}
}

@misc{BonnefontRajamakiVargas,
  author = {Bonnefont, Benjamin and Rajam{\"a}ki, Hermanni and Vargas, Vincent},
  title = {{Fourier} dimension of imaginary {Gaussian} multiplicative chaos},
  howpublished = {arXiv preprint arXiv:2512.19441},
  year = {2025},
  url = {https://arxiv.org/abs/2512.19441}
}

@article{BrunaNagelWainger,
  author = {Bruna, Joaquim and Nagel, Alexander and Wainger, Stephen},
  title = {Convex hypersurfaces and {Fourier} transforms},
  journal = {Ann. of Math. (2)},
  volume = {127},
  number = {2},
  year = {1988},
  pages = {333--365},
  url = {https://doi.org/10.2307/2007057}
}

@misc{CaiChenFangGuo2026,
  author = {Cai, Yin and Chen, Bonan and Fang, Xiang and Guo, Feng},
  title = {Critical {Gaussian} multiplicative chaos on the circle is {Rajchman}},
  howpublished = {arXiv preprint arXiv:2608.28328},
  year = {2026},
  url = {https://arxiv.org/abs/2608.28328}
}

@misc{CaiFangQu2026,
  author = {Cai, Yin and Fang, Xiang and Qu, Hongdou},
  title = {Exact {Fourier} dimensions of dyadic {Mandelbrot} cascades on curves of nonvanishing curvature under minimal integrability},
  howpublished = {arXiv preprint arXiv:2606.11758},
  year = {2026},
  url = {https://arxiv.org/abs/2606.11758}
}

@article{CarnovaleFraserdeOrellana2024,
  author = {Carnovale, Marc and Fraser, Jonathan M. and de Orellana, Ana E.},
  title = {Obtaining the {Fourier} spectrum via {Fourier} coefficients},
  journal = {Proc. Amer. Math. Soc.},
  volume = {154},
  year = {2026},
  number = {5},
  pages = {2005--2017},
  url = {https://doi.org/10.1090/proc/17610}
}

@misc{ChenHanQiuWang,
  author = {Chen, Xinxin and Han, Yong and Qiu, Yanqi and Wang, Zipeng},
  title = {Harmonic analysis of {Mandelbrot} cascades---in the context of vector-valued martingales},
  howpublished = {arXiv preprint arXiv:2409.13164},
  year = {2024},
  url = {https://arxiv.org/abs/2409.13164}
}

@misc{ChenLinQiu,
  author = {Chen, Yukun and Lin, Zhaofeng and Qiu, Yanqi},
  title = {Exact values of {Fourier} dimensions of {Gaussian} multiplicative chaos on high dimensional torus},
  howpublished = {arXiv preprint arXiv:2507.23494},
  year = {2025},
  url = {https://arxiv.org/abs/2507.23494}
}

@article{ChenLiSuomala,
  author = {Chen, Changhao and Li, Bing and Suomala, Ville},
  title = {{Fourier} dimension of {Mandelbrot} multiplicative cascades},
  journal = {Comm. Math. Phys.},
  volume = {406},
  year = {2025},
  number = {8},
  pages = {Paper No.~182, 15},
  url = {https://doi.org/10.1007/s00220-025-05354-x}
}

@article{DingRoyZeitouni,
  author = {Ding, Jian and Roy, Rishideep and Zeitouni, Ofer},
  title = {Convergence of the centered maximum of log-correlated {Gaussian} fields},
  journal = {Ann. Probab.},
  volume = {45},
  number = {6A},
  year = {2017},
  pages = {3886--3928},
  url = {https://doi.org/10.1214/16-AOP1152}
}

@article{DuplantierSheffield2011,
  author = {Duplantier, Bertrand and Sheffield, Scott},
  title = {{Liouville} quantum gravity and {KPZ}},
  journal = {Invent. Math.},
  volume = {185},
  number = {2},
  year = {2011},
  pages = {333--393},
  url = {https://doi.org/10.1007/s00222-010-0308-1}
}

@book{Falconer1997,
  author = {Falconer, Kenneth},
  title = {Techniques in fractal geometry},
  publisher = {John Wiley \& Sons, Ltd.},
  address = {Chichester},
  year = {1997},
  url = {https://kennethfalconer.github.io/papers.html}
}

@book{Falconer2014,
  author = {Falconer, Kenneth},
  title = {Fractal geometry: mathematical foundations and applications},
  edition = {Third},
  publisher = {John Wiley \& Sons, Ltd.},
  address = {Chichester},
  year = {2014},
  url = {https://www.wiley.com/en-us/Fractal+Geometry:+Mathematical+Foundations+and+Applications,+3rd+Edition-p-9781118762868}
}

@article{FalconerJin2019,
  author = {Falconer, Kenneth and Jin, Xiong},
  title = {Exact dimensionality and projection properties of {Gaussian} multiplicative chaos measures},
  journal = {Trans. Amer. Math. Soc.},
  volume = {372},
  number = {4},
  year = {2019},
  pages = {2921--2957},
  url = {https://doi.org/10.1090/tran/7776}
}

@article{Fraser2024,
  author = {Fraser, Jonathan M.},
  title = {The {Fourier} spectrum and sumset type problems},
  journal = {Math. Ann.},
  volume = {390},
  number = {3},
  year = {2024},
  pages = {3891--3930},
  url = {https://doi.org/10.1007/s00208-024-02843-7}
}

@article{GarbanVargas,
  author = {Garban, Christophe and Vargas, Vincent},
  title = {Harmonic analysis of {Gaussian} multiplicative chaos on the circle},
  journal = {Probab. Theory Related Fields},
  volume = {195},
  year = {2026},
  number = {3-4},
  pages = {1275--1297},
  url = {https://doi.org/10.1007/s00440-026-01497-7}
}

@misc{GorodetskyWong,
  author = {Gorodetsky, Ofir and Wong, Mo Dick},
  title = {On the limiting distribution of sums of random multiplicative functions},
  howpublished = {arXiv preprint arXiv:2508.12956},
  year = {2025},
  url = {https://arxiv.org/abs/2508.12956}
}

@article{Harper2020,
  author = {Harper, Adam J.},
  title = {Moments of random multiplicative functions, {I}: low moments, better than squareroot cancellation, and critical multiplicative chaos},
  journal = {Forum Math. Pi},
  volume = {8},
  year = {2020},
  pages = {e1, 95},
  url = {https://doi.org/10.1017/fmp.2019.7}
}

@article{Herz1962,
  author = {Herz, Carl S.},
  title = {{Fourier} transforms related to convex sets},
  journal = {Ann. of Math. (2)},
  volume = {75},
  number = {1},
  year = {1962},
  pages = {81--92},
  url = {https://doi.org/10.2307/1970421}
}

@book{HytonenVanNeervenVeraarWeis2016,
  author = {Hyt{\"o}nen, Tuomas and van Neerven, Jan and Veraar, Mark and Weis, Lutz},
  title = {Analysis in {Banach} spaces. {Vol. I}. {Martingales} and {Littlewood--Paley} theory},
  series = {Ergebnisse der Mathematik und ihrer Grenzgebiete. 3. Folge},
  volume = {63},
  publisher = {Springer},
  address = {Cham},
  year = {2016},
  url = {https://doi.org/10.1007/978-3-319-48520-1}
}

@article{JunnilaSaksmanWebb2019,
  author = {Junnila, Janne and Saksman, Eero and Webb, Christian},
  title = {Decompositions of log-correlated fields with applications},
  journal = {Ann. Appl. Probab.},
  volume = {29},
  number = {6},
  year = {2019},
  pages = {3786--3820},
  url = {https://doi.org/10.1214/19-AAP1492}
}

@article{Kahane1985,
  author = {Kahane, Jean-Pierre},
  title = {Sur le chaos multiplicatif},
  journal = {Ann. Sci. Math. Qu{\'e}bec},
  volume = {9},
  number = {2},
  year = {1985},
  pages = {105--150},
  url = {https://www.labmath.uqam.ca/~annales/volumes/09-2/PDF/105-150.pdf}
}

@misc{LinQiuTanI,
  author = {Lin, Zhaofeng and Qiu, Yanqi and Tan, Mingjie},
  title = {Harmonic analysis of multiplicative chaos {Part I}: the proof of {Garban--Vargas} conjecture for {1D GMC}},
  howpublished = {arXiv preprint arXiv:2411.13923},
  year = {2024},
  url = {https://arxiv.org/abs/2411.13923}
}

@misc{LinQiuTanII,
  author = {Lin, Zhaofeng and Qiu, Yanqi and Tan, Mingjie},
  title = {Harmonic analysis of multiplicative chaos {Part II}: a unified approach to {Fourier} dimensions},
  howpublished = {arXiv preprint arXiv:2505.03298},
  year = {2025},
  url = {https://arxiv.org/abs/2505.03298}
}

@book{Mattila2015,
  author = {Mattila, Pertti},
  title = {Fourier analysis and {Hausdorff} dimension},
  series = {Cambridge Studies in Advanced Mathematics},
  volume = {150},
  publisher = {Cambridge University Press},
  address = {Cambridge},
  year = {2015},
  url = {https://doi.org/10.1017/CBO9781316227619}
}

@article{NajnudelPaquetteSimm,
  author = {Najnudel, Joseph and Paquette, Elliot and Simm, Nick},
  title = {Secular coefficients and the holomorphic multiplicative chaos},
  journal = {Ann. Probab.},
  volume = {51},
  number = {4},
  year = {2023},
  pages = {1193--1248},
  url = {https://doi.org/10.1214/22-AOP1616}
}

@misc{NajnudelPaquetteSimmVu,
  author = {Najnudel, Joseph and Paquette, Elliot and Simm, Nick and Vu, Truong},
  title = {The {Fourier} coefficients of the holomorphic multiplicative chaos in the limit of large frequency},
  howpublished = {arXiv preprint arXiv:2502.14863},
  year = {2025},
  url = {https://arxiv.org/abs/2502.14863}
}

@article{Polyakov1981,
  author = {Polyakov, A. M.},
  title = {Quantum geometry of bosonic strings},
  journal = {Phys. Lett. B},
  volume = {103},
  number = {3},
  year = {1981},
  pages = {207--210},
  url = {https://doi.org/10.1016/0370-2693(81)90743-7}
}

@book{RevuzYor,
  author = {Revuz, Daniel and Yor, Marc},
  title = {Continuous martingales and {Brownian} motion},
  series = {Grundlehren der Mathematischen Wissenschaften},
  volume = {293},
  edition = {Third},
  publisher = {Springer-Verlag},
  address = {Berlin},
  year = {1999},
  url = {https://doi.org/10.1007/978-3-662-06400-9}
}

@article{RhodesVargas2014,
  author = {Rhodes, R{\'e}mi and Vargas, Vincent},
  title = {Gaussian multiplicative chaos and applications: a review},
  journal = {Probab. Surv.},
  volume = {11},
  year = {2014},
  pages = {315--392},
  url = {https://doi.org/10.1214/13-PS218}
}

@article{RyouSuomala2026,
  author = {Ryou, Donggeun and Suomala, Ville},
  title = {Fourier dimension of {Mandelbrot} cascades on planar curves},
  journal = {Probab. Theory Related Fields},
  year = {2026},
  url = {https://doi.org/10.1007/s00440-026-01528-3}
}

@article{SaksmanWebb2020,
  author = {Saksman, Eero and Webb, Christian},
  title = {The {Riemann} zeta function and {Gaussian} multiplicative chaos: statistics on the critical line},
  journal = {Ann. Probab.},
  volume = {48},
  number = {6},
  year = {2020},
  pages = {2680--2754},
  url = {https://doi.org/10.1214/20-AOP1433}
}

@article{Schulz1991,
  author = {Schulz, Helmut},
  title = {Convex hypersurfaces of finite type and the asymptotics of their {Fourier} transforms},
  journal = {Indiana Univ. Math. J.},
  volume = {40},
  number = {4},
  year = {1991},
  pages = {1267--1275},
  url = {https://doi.org/10.1512/iumj.1991.40.40056}
}

@article{ShmerkinSuomala,
  author = {Shmerkin, Pablo and Suomala, Ville},
  title = {Spatially independent martingales, intersections, and applications},
  journal = {Mem. Amer. Math. Soc.},
  volume = {251},
  number = {1195},
  year = {2018},
  pages = {v+102},
  url = {https://doi.org/10.1090/memo/1195}
}

@article{SoundararajanZaman,
  author = {Soundararajan, Kannan and Zaman, Asif},
  title = {A model problem for multiplicative chaos in number theory},
  journal = {Enseign. Math.},
  volume = {68},
  number = {3--4},
  year = {2022},
  pages = {307--340},
  url = {https://doi.org/10.4171/LEM/1031}
}

@misc{Verreault2026,
  author = {Verreault, William},
  title = {Almost sure upper bound for sums of random multiplicative functions and critical chaos},
  howpublished = {arXiv preprint arXiv:2608.21354},
  year = {2026},
  url = {https://arxiv.org/abs/2608.21354}
}

\end{document}